\documentclass[11pt]{article}

\usepackage{amsmath, amsthm, amssymb}
\usepackage{amsfonts}
\usepackage{epsfig}
\usepackage{fullpage}

\newtheorem{theorem}{Theorem}[section]
\newtheorem{proposition}[theorem]{Proposition}
\newtheorem{lemma}[theorem]{Lemma}
\newtheorem{corollary}[theorem]{Corollary}
\newtheorem{definition}[theorem]{Definition}
\newcounter{Examplecount}
\newcommand\beq{\begin{equation}}
\newcommand\eeq{\end{equation}}
\newcommand\bce{\begin{center}}
\newcommand\ece{\end{center}}
\newcommand\bea{\begin{eqnarray}}
\newcommand\eea{\end{eqnarray}}
\newcommand\ba{\begin{array}}
\newcommand\ea{\end{array}}
\newcommand\ben{\begin{enumerate}}
\newcommand\een{\end{enumerate}}
\newcommand\bit{\begin{itemize}}
\newcommand\eit{\end{itemize}}
\newcommand\brr{\begin{array}}
\newcommand\err{\end{array}}
\newcommand\bt{\begin{tabular}}
\newcommand\et{\end{tabular}}

\newcommand\nn{\nonumber}

\newcommand\ms{\medskip}
\newcommand\ul{\underline}

\renewcommand\S{{\mathcal S}}
\newcommand\Al{\operatorname{Allow}}

\newcommand\Pat{\operatorname{Pat}}

\newcommand{\W}{W}
\newcommand{\B}{B}
\newcommand{\BV}{\hat{b}}
\newcommand{\BW}{b}
\newcommand\Saw{M} 
\newcommand\WW{{\mathcal W}}
\def\red{\operatorname{st}}
\newcommand\des{\operatorname{des}}
\newcommand\hp{\hat\pi}
\newcommand\T{{\mathcal T}}
\providecommand{\fr}[1]{\{#1\}}
\providecommand{\fl}[1]{\lfloor#1\rfloor}

\newcommand{\lessv}{\prec}
\newcommand{\lesseqv}{\preccurlyeq}
\newcommand{\lessw}{\vartriangleleft}
\newcommand{\lesseqw}{\trianglelefteq}
\newcommand\G{\Gamma}
\providecommand{\andon}[1]{#1\rightarrow} 

\title{Permutations and $\beta$-shifts}
\author{Sergi Elizalde  \\ \\
Department of Mathematics\\ Dartmouth College \\ Hanover, NH 03755}
\date{}

\begin{document}

\maketitle
\begin{abstract}
Given a real number $\beta>1$, a permutation $\pi$ of length $n$ is realized by the $\beta$-shift if there is some $x\in[0,1]$ such that the relative order of the sequence $x,f(x),\dots,f^{n-1}(x)$,
where $f(x)$ is the factional part of $\beta x$, is the same as that of the entries of $\pi$.
Widely studied from such diverse fields as number theory and automata theory, $\beta$-shifts are prototypical examples one-dimensional chaotic dynamical systems.
When $\beta$ is an integer, permutations realized by shifts where studied in~\cite{Elishifts}. In this paper we generalize some of the results to arbitrary $\beta$-shifts.
We describe a method to compute, for any given permutation~$\pi$, the smallest $\beta$ such that $\pi$ is realized by the $\beta$-shift. We also give a way to
determine the length of the shortest forbidden (i.e., not realized) pattern of an arbitrary $\beta$-shift.
\end{abstract}

\section{Introduction}

Forbidden order patterns in piecewise monotone maps on one-dimensional intervals are a powerful tool to distinguish random from deterministic time series.
This contrasts with the fact that, from the viewpoint of symbolic dynamics, chaotic maps are able to produce any symbol pattern,
and for this reason they are used in practice to generate pseudo-random sequences.
However, this is no longer true when one considers order patterns instead, as shown in~\cite{AEK,Bandt}. From now on, we will use the term patterns to refer to order patterns.

The allowed patterns of a map on a one-dimensional interval are the permutations given by the
relative order of the entries in the finite sequences (usually called orbits) obtained by successively iterating the map,
starting from any point in the interval. For any fixed piecewise monotone map, there are some
permutations that do not appear in any orbit. These are called the forbidden patterns of the map.
Understanding the forbidden patterns of chaotic maps is important because the absence of these patterns is what distinguishes
sequences generated by chaotic maps from random sequences.

Determining the allowed and forbidden patterns of a given map is a difficult problem in general. The only non-trivial family of maps
for which the sets of allowed patterns have been characterized are shift maps. The first results in this direction are found
in~\cite{AEK}, and a characterization and enumeration of the allowed patterns of shift maps appears in~\cite{Elishifts}.
For another family, the so-called logistic map, a few basic properties of their set of forbidden patterns have been studied~\cite{Eliu}.

The focus of this paper are the allowed and forbidden patterns
of $\beta$-shifts, which are a natural generalization of shifts.
The combinatorial description of $\beta$-shifts is more elaborate than that of shifts, yet still simple enough for $\beta$-shifts
to be amenable to the study of their allowed patterns. At the same time, $\beta$-shifts are good prototypes of chaotic maps because they
exhibit important properties of low-dimensional chaotic dynamical systems, such sensitivity to initial conditions, strong mixing, and a dense set of periodic points.
 The origin of $\beta$-shifts lies in the study of expansions of
real numbers in an arbitrary real base $\beta>1$, which were introduced by R\'enyi~\cite{Ren}. Measure-theoretic properties of $\beta$-shifts
and their connection to these expansions have been extensively studied in the literature (see for example~\cite{Bla,Hof,Par,Schm}).
For instance, it is known that the base-$\beta$ expansion of $\beta$ itself determines the symbolic dynamics of the corresponding $\beta$-shift.
Finally, $\beta$-shifts have also been considered in computability theory~\cite{Sim}.

Related to the study of the allowed patterns of $\beta$-shifts, we are interested in the problem of determining,
for a given permutation $\pi$, what is the largest $\beta$ such that $\pi$ is a forbidden pattern of the $\beta$-shift.
We call this parameter the shift-complexity of the permutation. Putting technical details aside, this problem is equivalent to finding the smallest $\beta$ such
that $\pi$ is realized by (i.e., is an allowed pattern of) the $\beta$-shift

In Section~\ref{sec:background} we formally define allowed and forbidden patterns of maps,
and we describe shifts and $\beta$-shifts from a combinatorial perspective.
In Section~\ref{sec:wordstat} we introduce two relevant real-valued statistics on words. In Section~\ref{sec:shiftcomplexity} we study some properties of the domain of $\beta$-shifts, and we define shift-complexity.
Sections~\ref{sec:perm2word} and~\ref{sec:Bpi} explain how to determine the shift-complexity of a given permutation~$\pi$, by expressing
this parameter as a root of a certain polynomial whose coefficients depend on $\pi$ in a non-trivial way. In Section~\ref{sec:examples} we give examples of the usage of our method for particular permutations.
Finally, in Section~\ref{sec:shortest} we study the problem of finding, for given $\beta$, the shortest forbidden pattern of the $\beta$-shift.

\section{Background and notation}\label{sec:background}

Let $[n]=\{1,2,\dots,n\}$, and let $\S_n$ be the set of permutations of $[n]$. In the rest of the paper, the term permutation will always refer to an element of $\S_n$ for some $n$.
For a real number $x$, we use $\fl{x}$, $\lceil x\rceil$, and $\fr{x}$ to denote the floor, ceiling, and fractional part of $x$, respectively.
The fractional part of $x$ is also denoted by $x\mod1$ in some of the literature about shifts.

Most of the words considered in this paper will be infinite words over the alphabet $\{0,1,2,\dots\}$ that use only finitely many different letters.

\subsection{Allowed patterns of a map}

Let $X$ be a totally ordered set. Given a finite sequence $x_1,x_2,\dots,x_n$ of different elements of $X$, define its standardization $\red(x_1,x_2,\dots,x_n)$ to be the permutation of $[n]$ that is obtained by
replacing the smallest element in the sequence with $1$, the second smallest with $2$, and so on. For example, $\red(4,7,1,6.2,\sqrt{2})=35142$.

Let $f$ be a map $f:X\rightarrow X$. Given $x\in X$ and $n\ge1$, we define
$$\Pat(x,f,n)=\red(x,f(x),f^2(x),\dots,f^{n-1}(x)),$$
provided that there is no pair $1\le i< j\le n$ such that
$f^{i-1}(x)=f^{j-1}(x)$. If such a pair exists, then $\Pat(x,f,n)$ is not defined. When it is defined, then clearly $\Pat(x,f,n)\in\S_n$.

If $\pi\in\S_n$ and there is some $x\in I$ such that $\Pat(x,f,n)=\pi$, we say that $\pi$ is {\em realized} by $f$, or that $\pi$ is an {\em allowed pattern} of $f$.
The set of all permutations realized by $f$ is denoted by $\Al(f)=\bigcup_{n\ge1} \Al_n(f)$, where $$\Al_n(f)=\{\Pat(x,f,n): x\in X\}\subseteq\S_n.$$
The remaining permutations are called {\em forbidden patterns} of $f$.

\subsection{Shift maps}

Special cases of dynamical systems are shift systems. Shifts are interesting from a combinatorial perspective due to their simple definition, and at the same time they are
important dynamical systems because they
exhibit some key features of low-dimensional chaos, such as sensitivity to initial conditions, strong mixing, and a dense set of periodic
points.

For each $N\ge2$, let $\WW_N$ be the set of infinite words on the alphabet $\{0,1,\dots,N{-}1\}$, equipped with
the lexicographic order. The {\em shift on $N$ symbols} is defined to be the map
\bce\bt{cccc} $\Sigma_N:$ & $\WW_N$ & $\longrightarrow$ & $\WW_N$ \\  & $w_1w_2w_3\dots$ & $\mapsto$ & $w_2w_3w_4\dots$.\et\ece
For a detailed description of the associated dynamical system, see~\cite{AEK}.

According to the above definitions, we have for example that $\Pat(2102212210\dots,\Sigma_3,7)=4217536$, because the relative order of the successive shifts is
\bce\bt{rll}
$2102212210\dots$ & \quad & $4$ \\
$102212210\dots$ & \quad & $2$ \\
$02212210\dots$ & \quad & $1$ \\
$2212210\dots$ & \quad & $7$ \\
$212210\dots$ & \quad & $5$ \\
$12210\dots$ & \quad & $3$ \\
$2210\dots$ & \quad & $6$, \\ \et\ece
regardless of the entries in place of the dots.

Let $\Upsilon_N\subset\WW_N$ be the set of all words of the form $u(N{-}1)^\infty$, where $u$ is a finite word, and we use the notation $x^\infty=xxx\dots$.
Then $\WW_N\setminus\Upsilon_N$ is closed under shifts, and the map
\bce\bt{cccc} $\varphi:$ & $\WW_N\setminus\Upsilon_N$ & $\longrightarrow$ & $[0,1)$ \\  & $w_1w_2w_3\dots$ & $\mapsto$ & $\sum_{i\ge1} w_i N^{-i}$\et\ece
is an order-preserving bijection, also called an {\em order-isomorphism}. The map $\Saw_N=\varphi\circ\Sigma_N\circ\varphi^{-1}$ from $[0,1)$ to itself is the so-called {\em sawtooth map} $$\Saw_N(x)=\fr{Nx}$$
(see Figure~\ref{fig:sawtooth}). We say in this case that $\Sigma_N$ and $\Saw_N$ are order-isomorphic. As a consequence, $\Sigma_N$ and $\Saw_N$ have the same allowed and forbidden patterns.

\begin{figure}[hbt]
\bce\epsfig{file=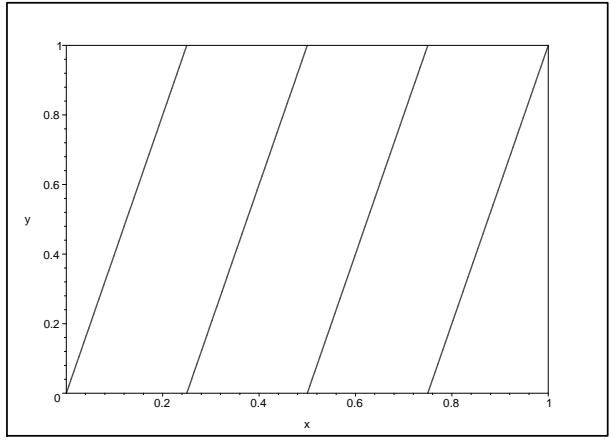,height=8cm,angle=-90} \caption{The sawtooth map $\Saw_4(x)=\fr{4x}$.\label{fig:sawtooth}}\ece
\end{figure}

Allowed and forbidden patterns of shifts (equivalently, sawtooth maps) were first studied in~\cite{AEK}, where
the authors prove the following result. 

\begin{theorem}[\cite{AEK}]\label{th:aek_minshift} For $N\ge2$, the shortest forbidden patterns of the shift $\Sigma_N$ have length $N+2$.
\end{theorem}
For example, the shortest forbidden patterns of $\Sigma_4$ are
$162534$, $435261$, $615243$, $342516$, $453621$, $324156$.
In fact, it was later shown in~\cite{Elishifts} that there are exactly six forbidden patterns of minimum length.
\begin{proposition}[\cite{Elishifts}]\label{prop:6special}
For every $N\ge 2$, the shortest forbidden patterns of $\Sigma_N$, which have length $n=N+2$, are
$\{\rho,\rho^R,\rho^C,\rho^{RC},\tau,\tau^C\}$, where
$$\rho\,=\, 1\, n\, 2\, (n{-}1)\, 3\, (n{-}2)\, \dots, \qquad \tau\, =\,\dots\, 4\, (n{-}1)\, 3\, n\, 2\, 1,$$
and $^R$ and $^C$ denote the reversal (obtained by reading the entries from right to left)
and complementation (obtained by replacing each entry $i$ with $n{+}1{-}i$) operations, respectively.
\end{proposition}

A formula is given in \cite{Elishifts} to compute, for any given permutation $\pi$, the minimum number of symbols needed in an alphabet in order for $\pi$ to be realized by a shift,
that is,
\beq\label{def:N}N(\pi):=\min\{N:\pi\in\Al(\Sigma_N)\}.\eeq
Table~\ref{tab:Npi} shows the values of $N(\pi)$ for all permutations of length up to~5.

\begin{table}[hbt]
$$\begin{array}{|c|c|}
\hline
  \pi & N(\pi) \\ \hline
  12, 21 & 2  \\ \hline
  132, 231, 312, 213; 321, 123 & 2  \\ \hline
  \begin{array}{c} 1234, 4321; 1243, 4312; 1324, 4231; 1342, 2431, 4213, 3124; \\ 1432, 2341, 4123, 3214; 2143, 3412, 2413, 3142\end{array} & 2 \\ \hline
  1423, 3241, 4132, 2314; 3421, 2134 & 3 \\ \hline
  \ba{c}
  12345, 12354, 12435, 12453, 12543, 13254, 13452, 13524, 13542, 14253, 14325, 15432,\\
  21543, 23451, 23541, 24135, 24513, 25314, 25413, 25431, 31254, 31425, 31542, 32154, \\
  34512, 35124, 35241, 35412, 41235, 41253, 41352, 42153, 42531, 43125, 43215, 45123, \\
  51234, 52341, 52413, 53124, 53142, 53214, 53412, 54123, 54213, 54231, 54312, 54321 \ea   & 2 \\ \hline
  \ba{c}
  12534, 13245, 13425, 14235, 14352, 14523, 14532, 15234, 15324, 15342, 15423, \\
  21345, 21354, 21435, 21453, 21534, 23154, 23415, 23514, 24153, 24315, 24351, \\
  24531, 25134, 25143, 25341, 31245, 31452, 31524, 32145, 32451, 32514, 32541, \\
  34125, 34152, 34215, 34521, 35142, 35214, 35421, 41325, 41523, 41532, 42135, \\
  42315, 42351, 42513, 43152, 43251, 43512, 45132, 45213, 45231, 45312, 45321, \\
  51243, 51324, 51342, 51432, 52134, 52143, 52314, 52431, 53241, 53421, 54132\ea   & 3 \\ \hline
  15243, 34251, 51423, 32415; 43521, 23145 &  4  \\ \hline
\end{array}$$
\caption{The values of $N(\pi)$ for permutations of length up to 5.}\label{tab:Npi}
\end{table}

The formula given to compute for $N(\pi)$ relies on a bijection between $\S_n$ and the set $\T_n$ of cyclic permutations of $[n]$ with a distinguished entry. For example, underlining the distinguished entry, we have $$\T_3=\{\ul{2}31,2\ul{3}1,23\ul{1},\ul{3}12,3\ul{1}2,31\ul{2}\}.$$
Given $\pi=\pi(1)\pi(2)\dots\pi(n)\in\S_n$, let $\hp\in\T_n$ be the permutation
whose cycle decomposition is $(\pi(1),\pi(2),\dots,\pi(n))$, with the entry $\pi(1)$ distinguished. For example, if $\pi=892364157$, then $$\hp=(\ul{8},9,2,3,6,4,1,5,7)=536174\ul{8}92.$$
For $\hp\in\T_n$, let $\des(\hp)$ denote the number of descents of the sequence that we get by deleting the distinguished entry from the one-line notation of $\hp$.
For example, $\des(536174\ul{8}92)=4$. With these definitions, we can now state the aforementioned formula for $N(\pi)$.
\begin{theorem}[\cite{Elishifts}]\label{thm:Npi2} Let $\pi\in\S_n$, and let $\hp$ be defined as above. Then $N(\pi)$ is given by
$$N(\pi)=1+\des(\hp)+\epsilon(\hp),$$
where $$\epsilon(\hp)=\begin{cases}
1 & \mbox{if } \hp(1) \mbox{ is the distinguished entry of $\hp$ and } \hp(2)\in\{1,n\}, \\
0 & \mbox{otherwise.}
\end{cases}$$
\end{theorem}
The distribution of the descent sets of cyclic permutations is studied in~\cite{Elides}.
The goal of the present paper is to obtain a formula to compute the analogue of $N(\pi)$ for the more general case of $\beta$-shifts, which we define next.

\subsection{$\beta$-shifts}

These maps are a natural generalization of shift maps, and have been extensively studied in the literature~\cite{Schm,Hof} from a
measure-theoretic perspective. Let us begin by defining their order-isomorphic counterparts on the unit interval, which we call $\beta$-sawtooth maps.
For any real number $\beta>1$, define the $\beta$-sawtooth map
\bce\bt{cccc} $\Saw_\beta:$ & $[0,1)$ & $\longrightarrow$ & $[0,1)$ \\  & $x$ & $\mapsto$ & $\fr{\beta x}$ \et\ece
(see Figure~\ref{fig:sawbeta}). In the rest of the paper we will assume, unless otherwise stated, that $\beta$ is a real number with $\beta>1$. Note that when $\beta$ is an integer we recover the definition of standard sawtooth maps.

\begin{figure}[hbt]
\bce\epsfig{file=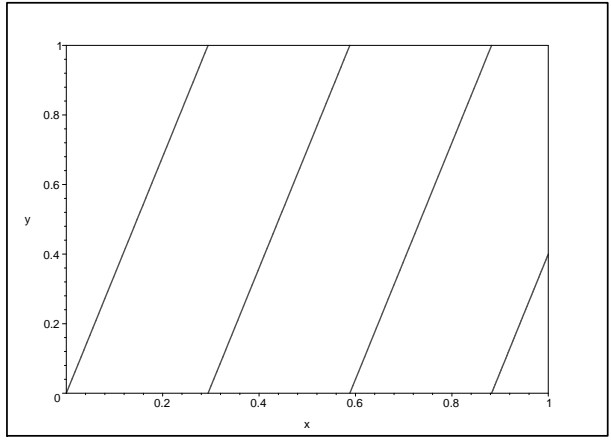,height=8cm,angle=-90} \caption{The $\beta$-sawtooth map $\Saw_\beta(x)$ for $\beta=3.4$.\label{fig:sawbeta}}\ece 
\end{figure}

To describe the corresponding map on infinite words, called the $\beta$-shift, let us first define its domain, $\W(\beta)$.
As shown by R\'enyi~\cite{Ren}, every nonnegative real number $x$ has a $\beta$-expansion
$$x=w_0+\frac{w_1}{\beta}+\frac{w_2}{\beta^2}+\cdots,$$
where
\bea
w_0 &=& \fl{x},\nn\\
w_1 &=& \fl{\beta \fr{x}},\label{eq:x2omega}\\
w_2 &=&  \fl{\beta\fr{\beta \fr{x}}},\nn \\
\dots\nn
\end{eqnarray}
This expansion has the property that $w_i\in\{1,2,\dots,\lceil\beta-1\rceil\}$ for $i\ge1$.
If $0\le x<1$, then $w_0=0$ and $w_i=\fl{\beta \Saw_\beta^{i-1}(x)}$ for $i\ge1$. 

Let $\W_0(\beta)$ be the set of infinite words $w=w_1w_2w_3\dots$ that are obtained in this way as $\beta$-expansions of numbers $x\in[0,1)$.
The lexicographic order (which will be denoted by $<$ throughout the paper) makes $\W_0(\beta)$ into a totally ordered set.
The map $[0,1)\rightarrow\W_0(\beta)$, $x\mapsto w$ is an order-isomorphism. For any $w=w_1w_2w_3\dots\in\W_0(\beta)$, we can recover $x\in[0,1)$ as $$x=\sum_{i\ge1} w_i \beta^{-i}.$$

Of particular interest is the $\beta$-expansion of $\beta$ itself, for which we use the notation
\beq\label{eq:betaexp}\beta=a_0+\frac{a_1}{\beta}+\frac{a_2}{\beta^2}+\cdots.\eeq
One can define $A_0=\beta$, and for $i\ge0$, $a_i=\fl{A_i}$ and $A_{i+1}=\beta(A_{i}-a_{i})$. Then it follows by induction that
\beq\label{eq:Ai} A_i=\beta^{i+1}-a_0\beta^i-a_1\beta^{i-1}-\dots-a_{i-1}\beta. \eeq

If $\beta$ is such that its $\beta$-expansion is finite, i.e., it has only finitely many nonzero terms $a_i$, we let $a_q$ be the last nonzero term of the expansion, so
$$\beta=a_0+\frac{a_1}{\beta}+\frac{a_2}{\beta^2}+\dots+\frac{a_q}{\beta^q},$$
and we let $y=(a_0a_1\dots a_{q-1}(a_q{-}1))^\infty$.
Define
\beq\label{eq:defW}\W(\beta)=\{w_1w_2w_3\ldots : w_kw_{k+1}w_{k+2}\ldots<a_0a_1a_2\ldots \mbox{ for all }k\ge1\}.\eeq
It follows from Parry~\cite{Par} that $\W_0(\beta)$ is precisely the set of words in $\W(\beta)$ that do not end in~$y$.
For example, if $\beta=N\in\mathbb{Z}$, then $\W(N)=\WW_N$, the set of infinite words on the alphabet $\{0,1,\dots,N-1\}$,
whereas $\W_0(N)$ does not include words ending in $(N{-}1)^\infty$. If $\beta=1+\sqrt{2}$, then $a_0a_1a_2\dots=210^\infty$, $\W(\beta)$ is the set of words over $\{0,1,2\}$ where every $2$ is followed by a $0$,
and $\W_0(\beta)$ is the set resulting from removing the words ending in $(20)^\infty$ from $\W(\beta)$. Clearly, if $\beta$ has an infinite $\beta$-expansion, then $\W(\beta)=\W_0(\beta)$.

We define the $\beta$-shift $\Sigma_\beta$ to be the map
\bce\bt{cccc} $\Sigma_\beta:$ & $\W(\beta)$ & $\longrightarrow$ & $\W(\beta)$ \\  & $w_1w_2w_3\dots$ & $\mapsto$ & $w_2w_3w_4\dots$.\et\ece

For $x\in[0,1)$, if $w\in\W_0(\beta)$ is the word given by the $\beta$-expansion of $x$, then $\Sigma_\beta(w)$ is the word given by the $\beta$-expansion of $\Saw_\beta(x)$. In particular, $\Saw_\beta$ and the restriction of $\Sigma_\beta$ to $\W_0(\beta)$ are order-isomorphic.
Besides, this restriction of the domain does not change the set of allowed patterns of $\Sigma_\beta$, and therefore $$\Al(\Sigma_\beta)=\Al(\Saw_\beta).$$

A well-studied problem is the connection between $\beta$-expansions and the ergodic properties of the corresponding $\beta$-shift (see~\cite{Schm} and references therein). In this paper, rather than the measure-theoretic
properties of $\beta$-shifts, we are concerned with their allowed and forbidden patterns.

\section{Two new statistics on words}\label{sec:wordstat}

In this section we define two real-valued statistics on words that will play a key role in studying the allowed patters of $\beta$-shifts. 
These statistics give rise to two weak orderings $\lessv$ and $\lessw$, that are related to but different from the lexicographic order $<$.
For an infinite word $w=w_1w_2\dots$, 
we use the notation $w_{\andon{i}}=w_iw_{i+1}\dots$ for $i\ge1$.

Throughout this section, $u,v,w,z$ denote words in $\WW_N$ for some arbitrary positive integer $N$. We define the series
$$f_w(\beta)=\frac{w_1}{\beta}+\frac{w_2}{\beta^2}+\dots+\frac{w_n}{\beta^n}+\cdots.$$
This series is convergent for $\beta>1$, and in this interval,
$$f_w'(\beta)=-\frac{w_1}{\beta^2}-\frac{2w_2}{\beta^3}-\dots<0,$$
assuming that $w\neq0^\infty$.
Since $\lim_{\beta\rightarrow\infty} f_w(\beta)=0$, it follows that there is a unique solution to $f_w(\beta)=1$ satisfying $\beta\ge1$.
Such value of $\beta$ will be denoted by $\BV(w)$. We define $\BV(0^\infty)=0$ by convention. Additionally, let $$\BW(w)=\sup_{i\ge1} \BV(w_{\andon{i}}).$$
Note that $\BV(w)\le \BW(w)\le N$.
The statistics $\BV$ and $\BW$ naturally define the following two weak orders on $\WW_n$.

\begin{definition}
We write \bit
\item $v\lessv w$\ if \ $\BV(v)<\BV(w)$,
\item $v\lessw w$\ if \ $\BW(v)<\BW(w)$.
\eit
\end{definition}

It follows from the definition of $\BV$ that the condition $\BV(v)<\BV(w)$ is equivalent to $f_v(\BV(w))<1$, and also to $f_w(\BV(v))>1$
(assuming $v\neq0^\infty$), since $f_v$ and $f_w$ are decreasing functions.
To denote the corresponding total preorders, we will write $v\lesseqv w$ if $\BV(v)\le\BV(w)$, and $v\lesseqw w$ if $\BW(v)\le\BW(w)$.
Recall that a total preorder is a binary relation that is reflexive, transitive, and has no incomparable pairs.

Note that $v<w$ lexicographically if and only if there is some $C$ such that $f_v(\beta)<f_w(\beta)$ for all $\beta>C$.
 The order relations $<,\lessv,\lessw$ are independent, in the sense that it is possible to find
a pair of elements for each predetermined combination of order relationships between them. For example, if $$u=21230^\infty, \ v=212310^\infty, \ w=2130010^\infty, \ z=3020^\infty,$$
we have $u<v<w<z$, $w\lessv u \lessv v \lessv z$, and $u\lessw w\lessw z\lessw v$.
Here are the values of $\BV$ and $\BW$ for these words:
$$\begin{array}{l|c|c}
& \BV & \BW \\ \hline
u=21230^\infty & 2.765123689\ldots & 3\hfill\\ \hline
v=212310^\infty & 2.776562146\ldots & 3.302775638\ldots \\ \hline
w=2130010^\infty & 2.761672412\ldots & 3.035744112\ldots \\ \hline
z=3020^\infty & 3.195823345\ldots & 3.195823345\ldots
\end{array}$$

In some cases, however, the are connections among the different orders, as shown in the lemmas below. Recall that $u,v,w,z$ denote words in $\WW_N$.

\begin{lemma}\label{lem:wi+1}
Let $i\ge1$. Then
$$w_{\andon{i+1}}\lesseqv w \ \Longleftrightarrow \  w \lesseqv w_1w_2\dots w_{i-1}(w_i{+}1).$$
The statement also holds substituting $\lessv$ for $\lesseqv$ on both sides.
\end{lemma}

\begin{proof}
Let $\beta=\BV(w)$.
The condition $w_{\andon{i+1}}\lesseqv w$ is equivalent to $f_{w_{\andon{i+1}}}(\beta)\le1$, that is
$$\frac{w_{i+1}}{\beta}+\frac{w_{i+2}}{\beta^2}+\cdots\le1.$$ Dividing by $\beta^i$ and adding $\frac{w_1}{\beta}+\dots+\frac{w_i}{\beta^i}$ to both sides, this inequality becomes
$$1=f_w(\beta)=\frac{w_{1}}{\beta}+\dots+\frac{w_{i}}{\beta^i}+\frac{w_{i+1}}{\beta^{i+1}}+\dots\le\frac{w_1}{\beta}+\dots+\frac{w_i}{\beta^i}+\frac{w_i+1}{\beta^{i+1}}=f_{w_1w_2\dots w_{i-1}(w_i{+}1)}(\beta),$$
which is equivalent to $w\lesseqv w_1w_2\dots w_{i-1}(w_i{+}1)$.

The corresponding statement for $\lessv$ is proved analogously, substituting $<$ for $\le$.
\end{proof}

\begin{lemma}\label{lem:compare}
Suppose that $v\lessv w$ and $v>w$, and let $i$ be the first position where the two words differ. Then $w\lessv w_{\andon{i+1}}$.\\
The statement also holds replacing the two occurrences of $\lessv$ with $\lesseqv$.
\end{lemma}

\begin{proof}
Since $w_j=v_j$ for $1\le j<i$ and $w_i<v_i$, each letter in $w_1w_2\dots w_{i-1}(w_i{+}1)$ is no greater than the corresponding letter in $v$, so $w_1w_2\dots w_{i-1}(w_i{+}1)\lesseqv v$.
Now suppose for contradiction that $w_{\andon{i+1}}\lesseqv w$. By Lemma~\ref{lem:wi+1}, we have $w \lesseqv w_1w_2\dots w_{i-1}(w_i{+}1)\lesseqv v$, contradicting that $v\lessv w$.

To prove the analogous statement for $\lesseqv$, we suppose that $w_{\andon{i+1}}\lessv w$, and then Lemma~\ref{lem:wi+1} implies that $w \lessv w_1w_2\dots w_{i-1}(w_i{+}1)\lesseqv v$, contradicting that $v\lesseqv w$.\
\end{proof}

\begin{lemma}\label{lem:z}
Suppose that $w_{\andon{j}}\le z$ for all $j$. Then $w_{\andon{j}}\lesseqv z$ for all $j$.
\end{lemma}

\begin{proof}
Let $\beta=\BV(z)$. The case $\beta=1$ is trivial, so we assume that $\beta>1$.
For each $j$, let $$s_j=\frac{w_{j}}{\beta}+\frac{w_{j+1}}{\beta^{2}}+\cdots.$$
If $j$ is such that $w_{\andon{j}}<z$, let $k_j-1\ge0$ be the length of the initial segment in which $w_{\andon{j}}$ an $z$ agree, and let $x_j$ be the first letter of $w_{\andon{j}}$ where they disagree.
In other words, $w_{\andon{j}}=z_1\dots z_{k_j-1}x_j w_{\andon{j+k_j}}$, with $x_j\le z_{k_j}-1$.
We then have
\beq\label{eq:sj}s_j\le\frac{z_1}{\beta}+\dots+\frac{z_{k_j-1}}{\beta^{k_j-1}}+\frac{z_{k_j}-1}{\beta^{k_j}}+\frac{1}{\beta^{k_j}}s_{j+k_j}\le 1-\frac{1}{\beta^{k_j}}+\frac{1}{\beta^{k_j}}s_{j+k_j},\eeq
using that $\frac{z_1}{\beta}+\frac{z_2}{\beta^2}+\dots=1$ and thus its partial sums are no greater than $1$.

To prove the lemma, suppose for contradiction that there is $j$ such that $\BV(w_{\andon{j}})>\beta$, or equivalently, $s_j>1$. Let $\epsilon=s_j-1>0$.
By inequality~(\ref{eq:sj}), $s_{j+k_j}-1\ge \beta^{k_j}(s_j-1)$. In particular, $s_{j+k_j}>1$. Repeating the same argument,
and setting $j_0=j$ and $j_{i+1}=j_i+k_{j_i}$ for $i\ge0$, we get
$$s_{j_{i+1}}-1\ge \beta^{k_{j_i}}(s_{j_i}-1)\ge\dots\ge \beta^{k_{j_i}+\dots+k_{j_0}}(s_{j_0}-1)\ge \beta^{i+1}\epsilon,$$ since $k_{j_i}\ge1$ for all $i$.
This implies that $s_{j_i}$ can be arbitrarily big for large enough $i$, which contradicts the fact that
$$s_{j_i}\le \frac{N-1}{\beta}+\frac{N-1}{\beta^2}+\dots= \frac{N-1}{\beta-1}.$$
\end{proof}

The converse of Lemma~\ref{lem:z} does not hold in general. For example, if the first entry of $z$ is $0$, then the condition $w_{\andon{j}}\le z$ for all $j$ forces $w=0^\infty$, but there may be other words with $w_{\andon{j}}\lesseqv z$ for all $j$. Nevertheless, a stronger result holds when $z$ is the $\beta$-expansion of $\beta$ for some $\beta>1$.

\begin{lemma}\label{lem:a}\renewcommand{\labelenumi}{(\roman{enumi})}
Let $a=a_0a_1a_2\dots$ be the $\beta$-expansion of $\beta$, for some $\beta>1$.
\ben
\item If $v\lesseqv a$, then $v\le a$;
\item if $v\lessv a$, then $v<a$;
\item for every $k\ge1$, $a_{\andon{k}}\lessv a$ and $a_{\andon{k}}<a$.
\een
\end{lemma}

\begin{proof}
Equations~(\ref{eq:betaexp}) and~(\ref{eq:Ai}) imply that for any $i\ge0$,
\beq\label{eq:fai1} f_{a_{\andon{i+1}}}(\beta)=\beta^i(\beta-a_0-\frac{a_1}{\beta}-\dots-\frac{a_i}{\beta^i})=\beta^{i+1}-a_0\beta^i-a_1\beta^{i-1}-\dots-a_{i-1}\beta-a_i=A_i-a_i<1,\eeq
since $a_i=\fl{A_i}$.

To prove part~(i), suppose for contradiction that $v\lesseqv a$ and $a<v$. By Lemma~\ref{lem:compare}, we have that $a\lesseqv a_{\andon{i+1}}$, where $a_i$ is the first entry in $a$ that differs from the corresponding entry in $v$.
But then $f_{a_{\andon{i+1}}}(\beta)\ge1$, which contradicts equation~(\ref{eq:fai1}).
In part~(ii), the condition $v\lessv a$ eliminates the possibility of $v=a$, so we have $v<a$ in this case.
Finally, equation~(\ref{eq:fai1}) and the fact that $\BV(a)=\beta$ imply that $a_{\andon{k}}\lessv a$ for all $k\ge1$, so part~(iii) follows now from part~(ii).
\end{proof}

\section{The shift-complexity of a permutation}\label{sec:shiftcomplexity}

In this section we establish some properties of the domain $W(\beta)$ of the $\beta$-shift, and we define a real-valued statistic on permutations, which we call the shift-complexity.

\begin{proposition}\label{prop:incl} Let $1 < \beta \le \beta'$. Then $$\W(\beta) \subseteq \W(\beta').$$
\end{proposition}

\begin{proof}
Let $a=a_0a_1\dots$ be the $\beta$-expansion of $\beta$, and let $a'=a_0'a_1'\dots$ be the $\beta'$-expansion of~$\beta'$.
By the definition in equation~(\ref{eq:defW}), it is enough to show that $a\le a'$. Let $A_i$ be as defined in equation~(\ref{eq:Ai}), and let $A_i'$ be defined analogously for $\beta'$.
Suppose that the first entry where $a$ and $a'$ differ is $a_i\neq a_i'$. We claim that $A_j\le A_j'$ for $0\le j\le i$. This follows by induction since $A_0=\beta\le\beta'=A_0'$,
and if $A_j\le A_j'$ for some $j<i$, then $\fl{A_j}=a_j=a_j'=\fl{A_j'}$ implies that $A_j-a_j\le A_j'-a_j'$, so $A_{j+1}=\beta(A_j-a_j)\le \beta'(A_j'-a_j')=A_{j+1}'$.
But then $a_i=\fl{A_i}\le\fl{A_i'}=a_i'$, so $a_i<a_i'$ and we are done.
\end{proof}

The set $\W(\beta)$ is related to the set of words $w$ for which $\BW(w)$ is bounded by $\beta$.
The following result gives the precise relationship. 

\begin{proposition}\label{prop:WBW}
Let $\beta>1$, and let $a=a_0a_1\dots$ be the $\beta$-expansion of $\beta$. We have
\bea\label{eq:part1}
& \{w : \BW(w)<\beta\} \subseteq \W(\beta) \subseteq \{w : \BW(w)\le\beta\}, \\
\label{eq:part2}& \W(\beta)=\{w : \BW(w)\le \beta \mbox{ and } w_{\andon{k}}\neq a \ \ \forall k\},\\
\label{eq:part3}& \BW(w)=\inf\{\beta : w\in\W(\beta)\}.
\eea
\end{proposition}

\begin{proof}
Since $f_a(\beta)=1$, we have that $\BV(a)=\beta$.
The inclusions~(\ref{eq:part1}) are proved by the following sequence of implications:
\begin{multline*}\BW(w)<\beta \ \Rightarrow \ \BV(w_{\andon{k}})<\beta=\BV(a) \ \ \forall k \ \Leftrightarrow \ w_{\andon{k}}\lessv a \ \ \forall k\ \underset{\mbox{Lemma~\ref{lem:a}(ii)}}{\Longrightarrow} \ w_{\andon{k}}<a \ \ \forall k \ \Leftrightarrow \ w\in\W(\beta) \\
\Rightarrow \ w_{\andon{k}}\le a \ \ \forall k \ \underset{\mbox{Lemmas~\ref{lem:z},~\ref{lem:a}(i)}}{\Longleftrightarrow} \ w_{\andon{k}}\lesseqv a \ \ \forall k\
\Leftrightarrow \ \BV(w_{\andon{k}})\le\beta=\BV(a) \ \ \forall k\ \Leftrightarrow \ \BW(w)\le\beta.
\end{multline*}
From the above argument we also see that the words $w$ with $\BW(w)\le\beta$ that are not in $\W(\beta)$ are precisely those with $w_{\andon{k}}=a$ for some $k$, proving~(\ref{eq:part2}).
Finally, equation~(\ref{eq:part3}) follows immediately from~(\ref{eq:part1}).
\end{proof}

It is not hard to see that the two inclusions in~(\ref{eq:part1}) are always strict.
For example, when $\beta=N\ge2$ is an integer, we have seen that $\W(N)=\WW_N$;
in this case $\{w : \BW(w)\le N\}$ also contains the words of the form $w_1\dots w_r N0^\infty$, with $w_i\le N-1$ for $1\le i\le r$,
and $\{w : \BW(w)< N\}$ does not contain words ending in $(N{-}1)^\infty$ nor other words such as $(N{-}1)0(N{-}1)^20(N{-}1)^30(N{-}1)^4\dots$.
For $\beta=1+\sqrt{2}$, $\W(\beta)$ is the set of words over $\{0,1,2\}$ where every $2$ is followed by a $0$, and $\{w : \BW(w)\le \beta\}$ additionally contains
the words of the form $w_1\dots w_r 210^\infty$, where every $2$ in $w_1\dots w_r$ is followed by a $0$.

Note that the statement of Proposition~\ref{prop:incl} also follows from Proposition~\ref{prop:WBW}.
Because of Proposition~\ref{prop:incl}, it is clear from the definition of $\beta$-shifts that for $\beta<\beta'$, the restriction of $\Sigma_{\beta'}$ to $\W(\beta)$ is equal to $\Sigma_{\beta}$.
An immediate consequence of this is the following corollary.
In the rest of the paper, we will write $\Sigma$ instead of $\Sigma_\beta$ when it creates no confusion.

\begin{corollary}\label{cor:inclal} If $1< \beta \leq \beta'$, then $$\Al(\Sigma_\beta) \subseteq \Al(\Sigma_{\beta'}).$$ \end{corollary}

\begin{proof}
If $\pi \in \Al(\Sigma_\beta)$, there exists by definition a word $w \in W(\beta)$ such that $\Pat(w,\Sigma_\beta,n) = \pi$, where $n$ is the length of $\pi$. By Proposition~\ref{prop:incl},  $w \in \W(\beta')$,
and since $\Pat(w, \Sigma_{\beta'}, n) = \pi$, we see that $\pi \in \Al(\Sigma_{\beta'})$.
\end{proof}

Now we can give the key definition of this section. We call $\B(\pi)$ the {\em shift-complexity} of $\pi$.
\begin{definition}\label{def:Bpi} For any permutation $\pi$, let
$$\B(\pi)=\inf\{\beta:\pi \in \Al(\Sigma_{\beta})\}.$$
\end{definition}

Equivalently, $\B(\pi)$ is the supremum of the set of values $\beta$ such that $\pi$ is a forbidden pattern of $\Sigma_\beta$.
If we think of the $\beta$-shifts $\Sigma_\beta$ as a family of functions parameterized by $\beta$, then the values of
$\beta$ for which there is a permutation $\pi$ with $\B(\pi)=\beta$ correspond to {\em phase transitions} where the set of allowed patterns of $\Sigma_\beta$ changes.
The next result describes the relationship between the permutation statistic $\B$ and the word statistic $\BW$.

\begin{proposition}\label{prop:BBW}
For any $\pi\in\S_n$,
$$\B(\pi)=\inf\{\BW(w):\Pat(w,\Sigma,n)=\pi\}.$$
\end{proposition}

\begin{proof}
Let $\Gamma$ be the set of words $w$ such that $\Pat(w,\Sigma,n)=\pi$.
The right hand side of the above equation equals
\beq\label{eq:infs}\inf\{\BW(w):w\in \Gamma\}=\inf\{\beta : \W(\beta)\cap\Gamma\neq\emptyset\},\eeq
using the equation~(\ref{eq:part3}) from Proposition~\ref{prop:WBW}.
The condition $\W(\beta)\cap\Gamma\neq\emptyset$, which states that there is some $w\in\W(\beta)$ with $\Pat(w,\Sigma,n)=\pi$, is equivalent
to the condition $\pi\in\Al(\Sigma_{\beta})$ by definition, so the right hand side of equation~(\ref{eq:infs}) equals $$\inf\{\beta:\pi \in \Al(\Sigma_{\beta})\}=\B(\pi).$$
\end{proof}

\section{Computation of $\B(\pi)$: from permutations to words}\label{sec:perm2word}

Suppose we are given $\pi\in\S_n$ with $n\ge2$.
The goal of this section and the next one is to describe a method to compute the shift-complexity of $\pi$.
In the rest of the paper, we refer to the condition $\Pat(w, \Sigma, n) = \pi$ by saying that $w$ {\em induces} $\pi$.
In some cases we will be able to find a word $w$ inducing $\pi$ such that $\BW(w)$ is smallest for all such words; when this happens, $\B(\pi)=\BW(w)$ and the infimum in Proposition~\ref{prop:BBW}
is a minimum. In other cases we will find a sequence of words $w^{(m)}$ inducing $\pi$ where $\BW(w^{(m)})$ approaches $\B(\pi)$ as $m$ grows.

This section is devoted to finding a word $w$ or a sequence $w^{(m)}$ with the above properties. In Section~\ref{sec:Bpi} we show how to compute the values of the statistic $\BW$ on these words in order to obtain $\B(\pi)$.

Let $N=N(\pi)$ for the rest of this section. From the definitions, it is clear that $\B(\pi)\le N$, and that there is some word $z\in\W(N)=\WW_N$ that induces $\pi$. The explicit construction of such words~$z$ is given in~\cite{Elishifts}.
It is important to notice that to find words $w$ and $w^{(m)}$ as described above, it is enough to consider only words in $\WW_N$.
Indeed, if $\B(\pi)=N$ (we will later see in equation~(\ref{eq:BN}) that this case never happens, but cannot rule it out just yet),
then any $z\in\WW_N$ inducing $\pi$ satisfies $\BW(z)=\B(\pi)$, and we can just take $w=z$.
On the other hand, to deal with the case $\B(\pi)<N$, note that any word $z$ with $\BW(z)<N$ must be in $\WW_N$.
This applies to $z=w$ for any word $w$ satisfying $\BW(w)=\B(\pi)$, and also to $z=w^{(m)}$ for words in the above sequence, provided that $\BW(w^{(m)})$ is close enough to $\B(\pi)$.
For convenience, a word $w$ inducing $\pi$ and satisfying
\beq\label{eq:smallword} \BW(w)=\B(\pi) \ \mbox{ or } \ \B(\pi)\le\BW(w)<N \eeq
will be called a {\em small} word.
For words in $\WW_N$ inducing $\pi$ we can apply Corollary~2.13 from~\cite{Elishifts}, which we restate here.

\begin{proposition}[\cite{Elishifts}]\label{prop:determined}
Let $N=N(\pi)$ as above, and suppose that $z\in\WW_N$ induces $\pi$. 
Then the entries $z_1z_2\dots z_{n-1}$ are uniquely determined by $\pi$.
\end{proposition}

In the rest of this section, we let $\zeta=\zeta(\pi)=z_1z_2\dots z_{n-1}$ be the word defined in Proposition~\ref{prop:determined}. It follows from~\cite{Elishifts} that the entries of $\zeta$ can be computed as follows:
\bit
\item Write the sequence of (unassigned) variables $z_{\pi^{-1}(1)}z_{\pi^{-1}(2)}\cdots z_{\pi^{-1}(n)}$ in this order and remove $z_n$ from it.
\item For each pair $z_iz_j$ of adjacent entries in the sequence with $z_i$ to the left of $z_j$, insert a vertical bar between them if and only if $\pi(i+1)>\pi(j+1)$.
\item In the case that $\pi(n)=1$ and $\pi(n-1)=2$, insert a vertical bar before the first entry in the sequence (which is $z_{\pi^{-1}(2)}$ in this case).
\item Set each $z_i$ in the sequence to equal the number of vertical bars to its left.
\eit
For example, if $\pi=892364157\in\S_9$, the sequence with $z_n$ removed is $z_7z_3z_4z_6z_8z_5z_1z_2$, which becomes $z_7|z_3z_4|z_6z_8|z_5z_1|z_2$ after inserting the bars,
so $\zeta(\pi)=z_1z_2\dots z_8=34113202$.

It is shown in~\cite[Lemma~2.8]{Elishifts} that if $1\le i,j<n$ are such that $\pi(i)<\pi(j)$ and $\pi(i+1)>\pi(j+1)$, then the corresponding entries in $\zeta(\pi)$ satisfy $z_i<z_j$. This statement is logically equivalent to the following.
\begin{lemma}[\cite{Elishifts}]\label{lem:order}
 If $1\le i,j<n$ are such that $z_j\le z_i$ and $\pi(j)>\pi(i)$, then $\pi(j+1)>\pi(i+1)$ and $z_j=z_i$.
\end{lemma}
The conclusion $z_j=z_i$ is clear from the fact that if $z_j<z_i$, then $z_{\andon{j}}<z_{\andon{i}}$, contradicting $\pi(j)<\pi(i)$.
From Proposition~\ref{prop:determined} we see that \beq\label{eq:boundz} 0\le z_1,\dots,z_{n-1}\le N-1.\eeq
It is shown in~\cite{Elishifts} that if $\pi(n-1)=n-1$ and $\pi(n)=n$, then $0\le z_1,\dots,z_{n-1}\le N-2$.

Proposition~\ref{prop:determined} and the paragraph preceding it imply that for any small word $w$ (as defined by condition~(\ref{eq:smallword})) inducing $\pi$,
the first $n-1$ entries of $w$ are given by $\zeta=z_1z_2\dots z_{n-1}$. In the rest of this section we show how to find the remaining entries $w_{n+1}w_{n+2}\dots$.
We begin with an easy special case.

\begin{proposition}\label{prop:c1} Suppose that $\pi(n) = 1$. Let
\beq\label{eq:w0} w = \zeta 0^\infty.\eeq  Then $w$ induces $\pi$, and for any other word $v$ that induces $\pi$, we have $\BW(v)>\BW(w)$. In particular, $$\B(\pi)=\BW(w)=\BV(w_{\andon{\ell}}),$$ where $\ell=\pi^{-1}(n)$.
\end{proposition}
\begin{proof}
It is shown in~\cite{Elishifts} that $w$ induces $\pi$. Thus, noting that $w_{\andon{n}}=0^\infty$, we have that $w_{\andon{s}}<w_{\andon{\ell}}$ for all $s\neq\ell$.
Now Lemma~\ref{lem:z} with $z=w_{\andon{\ell}}$ implies that $w_{\andon{s}}\lesseqv w_{\andon{\ell}}$ for all $s$, so $\BW(w)=\BV(w_{\andon{\ell}})$.

On the other hand, by Proposition~\ref{prop:determined}, any other small word $v$ that induces $\pi$ must have $\zeta$ as a prefix, so $v_i\ge w_i$ for all $i$, with at least one inequality being strict for some $i\ge n$.
It follows that $\BW(v)\ge\BV(v_{\andon{\ell}})>\BV(w_{\andon{\ell}})=\BW(w)$ and that $\B(\pi)=\BW(w)$.
\end{proof}

From equation~(\ref{eq:boundz}) it follows that $w=\zeta 0^\infty\in\WW_N$ and $\BW(w)<N$.
Let us now consider the case where $n$ appears in $\pi$ to the right of the entry $\pi(n)-1$.

\begin{theorem}\label{th:findw1} Suppose that $c=\pi(n) \neq 1$. Let $k=\pi^{-1}(c-1)$ and $\ell=\pi^{-1}(n)$, and suppose that $\ell > k$.
Let
\beq\label{eq:w1}w = z_1z_2\cdots z_{n-1}\, z_{k} z_{k+1} \cdots z_{\ell-2}(z_{\ell-1}{+}1) 0^\infty.\eeq
Then $w$ induces $\pi$, and for any other word $v$ that induces $\pi$, we have $\BW(v)\ge\BW(w)$.
In particular, $$\B(\pi)=\BW(w)=\BV(w_{\andon{\ell}}).$$
\end{theorem}

In order to establish this result, we begin proving some facts about $w$ as defined in equation~(\ref{eq:w1}).
We know by equation~(\ref{eq:boundz}) that $z_{\ell-1}\le N-1$.
We claim that $z_{\ell-1}<N-1$.
Indeed, if $z_{\ell-1}=N-1$ and $\ell<n$, using that $z_\ell\le z_{\ell-1}$ and $\pi(\ell)>\pi(\ell-1)$, Lemma~\ref{lem:order} would imply that $\pi(\ell+1)>\pi(\ell)=n$,
which is a contradiction. And if $z_{\ell-1}=N-1$ and $\ell=n$, then any word $z$ starting with $\zeta$ and satisfying $z_{\andon{\ell-1}}<z_{\andon{\ell}}$ (a necessary condition for $z$ to induce $\pi$) would
need to have some entry $z_i\ge N$, contradicting the definition of $N$ and Proposition~\ref{prop:determined}.
Thus, $z_{\ell-1}<N-1$, from where $w\in\WW_N$ and, since $w$ ends with $0^\infty$, $\BW(w)<N$ in this case as well.

Lemmas~\ref{lem:nlargest} and~\ref{lem:noinbetween} below assume the notation from the statement of
Theorem~\ref{th:findw1}.

\begin{lemma}\label{lem:nlargest}
For all $s\neq \ell$, we have $w_{\andon{s}}<w_{\andon{\ell}}$.
\end{lemma}

\begin{proof}
For convenience, let $d=\ell-k-1$ and define $\pi(n+i)=\pi(k+i)$ for $1\le i\le d$. Note that $w_{n+i}=w_{k+i}=z_{k+i}$ for $0\le i<d$ and that $w_{n+d}=w_{k+d}+1=z_{\ell-1}+1$.

Suppose for contradiction that there is some $s\neq \ell$ such that $w_{\andon{\ell}}\le w_{\andon{s}}$. Note that we cannot have $w_{\andon{\ell}}=w_{\andon{s}}$ because the position of the last nonzero entry $z_{\ell-1}{+1}$ is different in the two words,
so we must have $w_{\andon{\ell}}< w_{\andon{s}}$. Note also that $s\le n+d$ (otherwise we would have $w_{\andon{s}}=0^\infty$) and thus $\pi(s)<n=\pi(\ell)$.
For any $1\le j\le n$ and $1\le r\le n+d$, let $T(j,r)$ be the statement
$$w_{\andon{j}}< w_{\andon{r}} \ \mbox{ and }\ \pi(j)>\pi(r).$$
Our assumption implies that $T(\ell,s)$ holds.

Suppose now that $T(j,r)$ holds for certain $j,r$. Consider the following cases.
\bit
\item If $j=n$, then $w_{\andon{n}}< w_{\andon{r}}$ and $\pi(n)=c>\pi(r)$. From the definition of $w$, we have $w_{\andon{k}}<w_{\andon{n}}$. This implies that $w_{\andon{k}}<w_{\andon{r}}$
and $\pi(k)=c-1\ge\pi(r)$. Note also that the last inequality must be strict, otherwise $r=k$, which is impossible because $w_{\andon{k}}<w_{\andon{r}}$. It follows that $T(k,r)$ holds in this case.
\item If $j<n$ and $r<n-1$, then the first letters of $w_{\andon{j}}$ and $w_{\andon{r}}$ are $z_j$ and $z_r$, respectively, and $z_j\le z_r$.
By Lemma~\ref{lem:order}, $\pi(j+1)>\pi(r+1)$ and $z_j=z_r$, from where it also follows that $w_{\andon{j+1}}< w_{\andon{r+1}}$. Thus $T(j+1,r+1)$ holds.
\item If $j<n$ and $r=n+i$ for some $0\le i<d$,
then the first letters of $w_{\andon{j}}$ and $w_{\andon{r}}$ are $z_j$ and $z_{k+i}$, respectively, and $z_j\le z_{k+i}$. Also, $\pi(j)>\pi(r)\ge\pi(k+i)$, with strict inequality only when $i=0$.
By Lemma~\ref{lem:order}, $\pi(j+1)>\pi(k+i+1)=\pi(r+1)$ and $z_j=z_{k+i}$, from where it also follows that $w_{\andon{j+1}}< w_{\andon{r+1}}$. Thus $T(j+1,r+1)$ holds.
\item If $j<n$ and $r=n+d$, then $w_{\andon{r}}=(z_{\ell-1}{+}1)0^\infty$, so the fact that $w_{\andon{j}}< w_{\andon{r}}$ implies that $w_j=z_j\le z_{\ell-1}$. We also have $\pi(j)>\pi(r)=\pi(k+d)=\pi(\ell-1)$.
By Lemma~\ref{lem:order}, $\pi(j+1)>\pi(\ell)=n$, which is a contradiction.
\eit
We have shown that for $r<n+d$,
$T(j,r)$ implies $T(j+1,r+1)$ if $j<n$, and it implies $T(k,r)$ if $j=n$. It follows that if $T(\ell,s)$ holds,
then $T(j,r+d)$ must hold for some $j$, but this leads to a contradiction,
concluding our proof.
\end{proof}

\begin{lemma}\label{lem:noinbetween}
There exists no $t$ such that $w_{\andon{k}}<w_{\andon{t}}<w_{\andon{n}}$.
\end{lemma}

\begin{proof}
Suppose that the result is false, so there is some $t$ such that
$$z_k z_{k+1}\cdots z_{n-1}\, z_{k} z_{k+1} \cdots z_{\ell-2}(z_{\ell-1}{+}1) 0^\infty <w_{\andon{t}}<z_{k} z_{k+1} \cdots z_{\ell-2}(z_{\ell-1}{+}1) 0^\infty.$$
It follows that $w_{\andon{t}}=z_{k} z_{k+1} \cdots z_{\ell-2} z_{\ell-1} w_{\andon{t+\ell-k}}$, and that
$$w_{\andon{\ell}}=z_\ell z_{\ell+1}\cdots z_{n-1}\, z_{k} z_{k+1} \cdots z_{\ell-2}(z_{\ell-1}{+}1) 0^\infty< w_{\andon{t+\ell-k}},$$
which contradicts Lemma~\ref{lem:nlargest}.
\end{proof}

\begin{proof}[Proof of Theorem~\ref{th:findw1}]
First we show that $w$ induces $\pi$, following the proof of Proposition~2.12 from~\cite{Elishifts}.
For $1\le i,j\le n$, let $S(i,j)$ be the statement $$\pi(i)<\pi(j) \ \mbox{ implies }\ w_{\andon{i}}<w_{\andon{j}}.$$ We want to prove $S(i,j)$ for all $1\le i,j\le n$ with $i\neq j$.
We do this considering three cases.

\bit \item {\it Case $i=n$}. Suppose that $\pi(n)<\pi(j)$, so in particular $j\ne k$. By Lemma~\ref{lem:noinbetween}, in order to prove that $w_{\andon{n}}<w_{\andon{j}}$
it is enough to show that $w_{\andon{k}}<w_{\andon{j}}$. Also, $\pi(n)<\pi(j)$ implies $\pi(k)<\pi(j)$. Thus, we have reduced $S(n,j)$ to $S(k,j)$.

\item {\it Case $j=n$}. Suppose that $\pi(i)<\pi(n)$. It is clear from the definition of $w$ that $w_{\andon{k}}<w_{\andon{n}}$. If $i=k$ we are done. If $i\ne k$, then $\pi(i)<\pi(n)$
implies that $\pi(i)<\pi(k)$, since $\pi(k)=\pi(n)-1$. So, if $S(i,k)$ holds, then $w_{\andon{i}}<w_{\andon{k}}<w_{\andon{n}}$, so $S(i,n)$ must hold as well.
We have reduced $S(i,n)$ to $S(i,k)$. Equivalently, $\neg S(i,n)\Rightarrow\neg S(i,k)$, where $\neg$ denotes negation.

\item {\it Case $i,j<n$}. Suppose that $\pi(i)<\pi(j)$. If $z_i<z_j$, then $w_{\andon{i}}<w_{\andon{j}}$ and we are done. If $z_i=z_j$, then we know by Lemma~\ref{lem:order} that
$\pi(i+1)<\pi(j+1)$. If we can show that $w_{\andon{i+1}}<w_{\andon{j+1}}$, then $w_{\andon{i}}=z_iw_{\andon{i+1}}<z_jw_{\andon{j+1}}=w_{\andon{j}}$.
So, we have reduced $S(i,j)$ to $S(i+1,j+1)$.
\eit

The above three cases show that for all $1\le i,j\le n-1$, $\neg S(i,j)\Rightarrow\neg S(g(i),g(j))$, where $g$ is defined for $1\le i\le n-1$ by
$$g(i)=\begin{cases} i+1 & \mbox{if } i<n-1, \\ k & \mbox{if } i=n-1.\end{cases}$$
Suppose now that for some $i,j$, $S(i,j)$ does not hold. Using the first two cases above, we can assume that $1\le i,j\le n-1$. Then, $S(g^q(i),g^q(j))$ fails for every $q\ge 1$.
Let $r$ be such that $\pi(r)$ is the maximum of $\pi(k),\pi(k+1),\dots,\pi(n-1)$. Let $q$ be such that $g^q(i)=r$
and $k\le g^q(j)\le n-1$.
Then $S(g^q(i),g^q(j))$ must hold because $\pi(g^q(i))\ge \pi(g^q(j))$. This is a contradiction, so we have proved that $w$ induces $\pi$.

To see that $\BW(w)=\BV(w_{\andon{\ell}})$, we use Lemmas~\ref{lem:nlargest} and~\ref{lem:z} to conclude that $w_{\andon{s}}\lesseqv w_{\andon{\ell}}$ for all $s$.

\ms

Assume now that some other word $v$ induces $\pi$. Let us show that $\BW(v)\ge\BW(w)$, which will also imply that $\B(\pi)=\BW(w)$. Suppose for contradiction that $\BW(v)<\BW(w)< N$.
By Proposition~\ref{prop:determined}, the first $n-1$ entries of $v$ are given by $\zeta=z_1z_2\dots z_{n-1}$, so we can write $v=\zeta y$.
Since $z_kz_{k+1}\dots z_{n-1}y=v_{\andon{k}}<v_{\andon{n}}=y$, we have $y>(z_kz_{k+1}\dots z_{n-1})^\infty$.
Consider the leftmost position where the words $y$ and $(z_kz_{k+1}\dots z_{n-1})^\infty$ differ. We can
assume that in that position, the difference between the corresponding entries is one, and that to the right of it $y$ has zeros only,
since this assumption cannot increase the value of $\BW(\zeta y)$. In other words,
$y = (z_{k} z_{k+1} \cdots z_{n-1})^t z_{k} z_{k+1} \cdots z_{i-1} (z_i{+}1) 0^\infty$ for some $t\ge0$ and $k\le i\le n-1$. For convenience, let $i'=t(n-k)+i$.
If $i'=\ell-1$, then $v=w$ and there is nothing to prove.
Assuming $i'\neq\ell-1$, we will find $j$ such that $w_{\andon{\ell}}\le v_{\andon{j}}$.

If $i'<\ell-1$, then $w_{\andon{\ell}}<v_{\andon{\ell}}$, and we take $j=\ell$.
If $i'>\ell-1$, let $j-1$ be the position of the rightmost copy of $z_{\ell-1}$ in $v$.
We then have $w_{\andon{\ell}}\le v_{\andon{j}}$.
In both cases, Lemma~\ref{lem:nlargest} implies that $$w_{\andon{s}}\le w_{\andon{\ell}}\le v_{\andon{j}}$$ for all $s$. By Lemma~\ref{lem:z}, $w_{\andon{s}}\lesseqv v_{\andon{j}}$ for all $s$, and thus
$\BW(w)\le \BV(v_{\andon{j}})\le \BW(v)$.
\end{proof}

We are left with the case where $n$ appears in $\pi$ to the left of the entry $\pi(n)-1$.

\begin{theorem} \label{th:findw2} Suppose that $c=\pi(n) \neq 1$. Let $k=\pi^{-1}(c-1)$ and $\ell=\pi^{-1}(n)$, and suppose that $\ell < k$.
Let $h$ be such that $\pi(h)$ is the maximum of $\pi(k+1),\pi(k+2),\dots,\pi(n)$. For each $m\ge0$, let
\beq\label{eq:w2}
w^{(m)} = z_1z_2\cdots z_{n-1}\, (z_{k} z_{k + 1} \cdots z_{n-1})^{m} z_k z_{k+1} \cdots z_{h-2}(z_{h-1}{+}1)0^\infty.
\eeq
Then $w^{(m)}$ induces $\pi$ for $m\ge \frac{n-2}{n-k}$, and for any word $v$ that induces $\pi$, there exists an $m_0$ such that $\BW(v)>\BW(w^{(m)})$ for $m\ge m_0$.
In particular, $$\B(\pi)=\lim_{m\rightarrow\infty}\BW(w^{(m)}).$$
Additionally, $\BW(w^{(m)})=\BV(w^{(m)}_{\andon{\ell}})$.
\end{theorem}

Before proving this result, let us recall equation~(\ref{eq:boundz}), and show that $z_{h-1}<N-1$ in this case.
Indeed, if $z_{h-1}=N-1$ and $h<n$, using that $z_h\le z_{h-1}$ and $\pi(h)>\pi(h-1)$, Lemma~\ref{lem:order} would imply that $\pi(h+1)>\pi(h)$,
which contradicts the choice of $h$. And if $z_{h-1}=N-1$ and $h=n$, then any word $z$ starting with $\zeta$ and satisfying $z_{\andon{h-1}}<z_{\andon{h}}$ (a necessary condition for $z$ to induce $\pi$) would
need to have some entry $z_i\ge N$, contradicting the definition of $N$ and Proposition~\ref{prop:determined}.
From the fact that $z_{h-1}<N-1$ we conclude that $w^{(m)}\in\WW_N$ and so $\BW(w^{(m)})< N$ (this inequality is strict because $w^{(m)}$ ends with $0^\infty$).

In Lemma~\ref{lem:nlargest2} below, the notation is the same as in the statement of
Theorem~\ref{th:findw2}. The result is analogous to Lemma~\ref{lem:nlargest}, and the proof uses very similar ideas.

\begin{lemma}\label{lem:nlargest2}
For all $m\ge0$ and $s\neq \ell$, we have $w^{(m)}_{\andon{s}}<w^{(m)}_{\andon{\ell}}$.
\end{lemma}

\begin{proof}
To simplify notation, let $w=w^{(m)}$ in this proof.
Let $R=n-1+m(n-k)+h-k$ the position of the last nonzero entry in $w$, so $w_R=z_{h-1}{+}1$.
For each $r$ with $n< r\le R$, which can be written uniquely as
$r=n+q(n-k)+i$ with $0\le q\le m$ and $0\le i\le n-1-k$,
define for convenience $\pi(r)=\pi(k+i)$ if $1\le i\le n-1-k$, and $\pi(r)=\pi(n)$ if $i=0$.
Note that $w_{r}=w_{k+i}=z_{k+i}$ for $n<r<R$.

Suppose for contradiction that there is some $s\neq \ell$ such that $w_{\andon{\ell}}\le w_{\andon{s}}$. Note that we cannot have $w_{\andon{\ell}}=w_{\andon{s}}$ because the position of the last nonzero entry $z_{h-1}{+}1$ is different in the two words, so we must have $w_{\andon{\ell}}< w_{\andon{s}}$. Note also that $s\le R$. Since $\ell<k$ and thus none of the extended values of $\pi$ equals $n$, we have $\pi(s)<n=\pi(\ell)$.
For any $1\le j\le n$ and $1\le r\le R$, let $T(j,r)$ be the statement
$$w_{\andon{j}}< w_{\andon{r}} \ \mbox{ and }\ \pi(j)>\pi(r).$$
Our assumption implies that $T(\ell,s)$ holds.

Suppose now that $T(j,r)$ holds for certain $j,r$. Consider the following cases.
\bit
\item If $j=n$, then $w_{\andon{n}}< w_{\andon{r}}$ and $\pi(n)=c>\pi(r)$. From the definition of $w=w^{(m)}$, we have $w_{\andon{k}}<w_{\andon{n}}$. This implies that $w_{\andon{k}}<w_{\andon{r}}$
and $\pi(k)=c-1\ge\pi(r)$. Note also that the last inequality must be strict, otherwise $r=k$, which is impossible because $w_{\andon{k}}<w_{\andon{r}}$. It follows that $T(k,r)$ holds in this case.
\item If $j<n$ and $r<n$, then the first letters of $w_{\andon{j}}$ and $w_{\andon{r}}$ are $z_j$ and $z_r$, respectively, and $z_j\le z_r$.
By Lemma~\ref{lem:order}, $\pi(j+1)>\pi(r+1)$ and $z_j=z_r$, from where it also follows that $w_{\andon{j+1}}< w_{\andon{r+1}}$. Thus $T(j+1,r+1)$ holds.
\item If $j<n$ and $n\le r<R$, we write $r=n+q(n-k)+i$ with $0\le q\le m$ and $0\le i\le n-1-k$ as before.
The first letters of $w_{\andon{j}}$ and $w_{\andon{r}}$ are $z_j$ and $z_{k+i}$, respectively, and $z_j\le z_{k+i}$. Also, $\pi(j)>\pi(r)\ge\pi(k+i)$, with the last inequality being strict only when $i=0$.
By Lemma~\ref{lem:order}, $\pi(j+1)>\pi(k+i+1)=\pi(r+1)$ and $z_j=z_{k+i}$, from where it also follows that $w_{\andon{j+1}}< w_{\andon{r+1}}$. Thus $T(j+1,r+1)$ holds.
\item If $j<n$ and $r=R$, then $w_{\andon{r}}=(z_{h-1}{+}1)0^\infty$, so the fact that $w_{\andon{j}}< w_{\andon{r}}$ implies that $w_j=z_j\le z_{h-1}$. We also have $\pi(j)>\pi(r)=\pi(h-1)$.
By Lemma~\ref{lem:order}, $\pi(j+1)>\pi(h)$, which, by the choice of $h$, can only hold if $j<k$.
\eit
We have shown that for $r<R$,
$T(j,r)$ implies $T(j+1,r+1)$ if $j<n$, and it implies $T(k,r)$ if $j=n$. It follows that if $T(\ell,s)$ holds,
then $T(j,R)$ must hold for some $j$ with $j\ge k$. But this leads to a contradiction, as shown in the last of the above cases, thus
concluding our proof.
\end{proof}

\begin{proof}[Proof of Theorem~\ref{th:findw2}]
To see that $w^{(m)}$ induces $\pi$ for $m$ large enough, we refer the reader to~\cite{Elishifts}. It is shown there that
the word $z_1z_2\cdots z_{n-1}\, (z_{k} z_{k + 1} \cdots z_{n-1})^{m} z_k z_{k+1} \cdots z_{h-1}(N{-}1)^\infty$ (which is denoted by $w_B(\pi)$ in~\cite{Elishifts}) induces $\pi$ when $m\ge\frac{n-2}{n-k}$.
The proof that $w^{(m)}$ induces $\pi$ is identical, the main ingredients being the fact that the word $z_{k} z_{k + 1} \cdots z_{n-1}$ is primitive, and the analogue of Lemma~\ref{lem:noinbetween} for this case.

If $m$ is such that $w^{(m)}$ induces $\pi$, we get from Lemmas~\ref{lem:nlargest2} and~\ref{lem:z} that $w^{(m)}_{\andon{s}}\lesseqv w^{(m)}_{\andon{\ell}}$ for all $s$, so $\BW(w^{(m)})=\BV(w^{(m)}_{\andon{\ell}})$.

Assume now that some word $v$ induces $\pi$, and suppose for contradiction that $\BW(v)\le\BW(w^{(m)})<N$ for arbitrarily large $m$. In particular
$v\in\WW_N$, so by Proposition~\ref{prop:determined}, $v$ must start with $\zeta$, so we can write $v=\zeta y$.
Since $z_kz_{k+1}\dots z_{n-1}y=v_{\andon{k}}<v_{\andon{n}}=y$, we have $y>(z_kz_{k+1}\dots z_{n-1})^\infty$.
Consider the leftmost position where these two words differ. We can
assume that in that position, the difference between the corresponding entries is one, and that to the right of it $y$ has zeros only,
since this assumption cannot increase the value of $\BW(v)$.
In other words, $y = (z_{k} z_{k+1} \cdots z_{n-1})^t z_{k} z_{k+1} \cdots z_{i-1} (z_i{+}1) 0^\infty$ for some $t\ge0$, $k\le i\le n-1$.
Now, taking any $m>t$, we have that $w^{(m)}_{\andon{\ell}}<v_{\andon{\ell}}=z_\ell z_{\ell+1} \dots z_{n-1}y$.
Using Lemmas~\ref{lem:nlargest2} and~\ref{lem:z}, we can argue as in the proof of Theorem~\ref{th:findw2} to show that
$\BW(w^{(m)})\le \BV(v_{\andon{\ell}})\le \BW(v)$.
However, to prove the strict inequality $\BW(w^{(m)})<\BW(v)$, we use the fact that $w^{(m)}_{\andon{\ell}}$ is the $\beta$-expansion of $\beta=\BV(w^{(m)}_{\andon{\ell}})$, which is proved in Lemma~\ref{lem:wlbetaexp} below.
This property of $w^{(m)}_{\andon{\ell}}$, combined with Lemma~\ref{lem:a}(i) and the fact that $w^{(m)}_{\andon{\ell}}< v_{\andon{\ell}}$, implies that $w^{(m)}_{\andon{\ell}}\lessv v_{\andon{\ell}}$.
We conclude that $$\BW(w^{(m)})=\BV(w^{(m)}_{\andon{\ell}})<\BV(v_{\andon{\ell}})\le \BW(v).$$
\end{proof}

Several examples of applications of the above results are given in Section~\ref{sec:examples}. Section~\ref{sec:Bpi} deals with the problem of computing $\BW(w)$ and $\lim_{m\rightarrow\infty} \BV(w^{(m)})$, where $w$ and $w^{(m)}$ are the above words. Let us first prove a property of these words.

\begin{lemma}\label{lem:wlbetaexp}
Let $c=\pi(n)$, $\ell=\pi^{-1}(n)$, and if $c\neq1$, let $k=\pi^{-1}(c-1)$. Let
$$u=\begin{cases} w_{\andon{\ell}}, & \mbox{where $w$ is given by equation~(\ref{eq:w0}), if $c=1$,} \\
w_{\andon{\ell}}, & \mbox{where $w$ is given by equation~(\ref{eq:w1}), if $c\neq1$ and $\ell>k$,} \\
w^{(m)}_{\andon{\ell}} & \mbox{for any fixed $m\ge0$, where $w^{(m)}$ is given by equation~(\ref{eq:w2}), if $c\neq1$ and $\ell<k$.}\end{cases}$$
Then $u$ is the $\beta$-expansion of $\beta=\BV(u)$.
\end{lemma}

\begin{proof}
It is proved in~\cite{Par,Schm} that a word $u=u_1u_2\dots$ is a $\beta$-expansion of some $\beta$ if and only if $u_{\andon{i}}<u$ for all $i>1$, and that in this case $\beta$ is unique (in fact, $\beta=\BV(u)$).
In the first of the three above cases, it is clear that $u$ has this property from its definition and the fact that $w$ induces $\pi$.
In the second and third cases, the fact that $u_{\andon{i}}<u$ for all $i>1$ follows from Lemmas~\ref{lem:nlargest} and~\ref{lem:nlargest2}, respectively.
\end{proof}

We end this section looking in more detail at the phase transitions where new patterns become allowed for $\beta$-shifts, and discussing the relationship between $\B(\pi)$ and $N(\pi)$.

\begin{proposition}\label{prop:open}
For every $\pi\in\S_n$,
$$\pi\notin \Al(\Sigma_{\B(\pi)}).$$ In particular, the infimum in Definition~\ref{def:Bpi} is never a minimum, and the shift-complexity of $\pi$ is the maximum $\beta$ such that $\pi$ is a forbidden pattern of $\Sigma_\beta$.
\end{proposition}

\begin{proof}
Let $\beta=\B(\pi)$.
The statement $\pi\notin \Al(\Sigma_{\beta})$ is equivalent to the fact that there is no word in $\W(\beta)$, the domain of $\Sigma_\beta$, that induces $\pi$. Any word $v$ that induces $\pi$ has $\BW(v)\ge\beta$.
If $\BW(v)>\beta$, then clearly $v\notin\W(\beta)$.
We will show that for every word $v$ inducing $\pi$ and with $\BW(v)=\beta$, there is some $j$ such that $v_{\andon{j}}$ is the $\beta$-expansion of $\beta$. This will imply that
$v\notin\W(\beta)$ by equation~(\ref{eq:part2}) from Proposition~\ref{prop:WBW}, concluding the proof.

Let $c,\ell,k$ be defined as in Lemma~\ref{lem:wlbetaexp}.
If $c=1$, then by Proposition~\ref{prop:c1} there is only one word $w$ inducing $\pi$ with $\BW(w)=\beta$, namely the word given by equation~(\ref{eq:w0}). By Lemma~\ref{lem:wlbetaexp},
$w_{\andon{\ell}}$ is then the $\beta$-expansion of $\beta$.

Suppose now that $c\neq1$ and $\ell>k$. The last part of the proof of Theorem~\ref{th:findw1} shows that if $w$ is given by equation~(\ref{eq:w1}) and $v$ is any other word inducing $\pi$,
then there is some $j$ such that $w_{\andon{\ell}}\le v_{\andon{j}}$. By Lemma~\ref{lem:wlbetaexp}, $w_{\andon{\ell}}$ is the $\beta$-expansion of $\beta$.
Thus, if $v_{\andon{j}}=w_{\andon{\ell}}$, then $v\notin\W(\beta)$.
If $w_{\andon{\ell}}<v_{\andon{j}}$, then Lemma~\ref{lem:a}(i) implies that $w_{\andon{\ell}}\lessv v_{\andon{j}}$, from where
$\beta=\BW(w)=\BV(w_{\andon{\ell}})<\BV(v_{\andon{j}})\le \BW(v)$, so $v\notin\W(\beta)$ in this case.

Finally, if $c\neq1$ and $\ell<k$, Theorem~\ref{th:findw2} states that for any word $v$ inducing $\pi$, we have $\BW(v)>\BW(w^{(m)})\ge \beta$ for $m$ large enough, where $w^{(m)}$ is given by equation~(\ref{eq:w1}), so
$v\notin\W(\beta)$ again.

The last sentence of the proposition is an easy consequence of Corollary~\ref{cor:inclal} and Definition~\ref{def:Bpi}.
\end{proof}

One can rephrase Proposition~\ref{prop:open} by stating that $\pi\in\Al(\Sigma_\beta)$ if and only if $\beta>\B(\pi)$.
It follows from this observation and the definition of $N(\pi)$ (see equation~(\ref{def:N})) that
\beq\label{eq:BN}N(\pi)=\fl{\B(\pi)}+1.\eeq

\section{Computation of $\B(\pi)$: the equations}\label{sec:Bpi}

In this section we find the shift-complexity of an arbitrary permutation $\pi$ by expressing it as the unique real root greater than one of a certain polynomial $P_\pi(\beta)$.
Given a finite word $u_1u_2\dots u_r$, define the polynomial
$$p_{u_1u_2\dots u_r}(\beta)=\beta^r-u_1\beta^{r-1}-u_2\beta^{r-2}-\dots-u_r.$$

\begin{theorem}\label{th:Bpi}
For any $\pi\in\S_n$ with $n\ge2$, let $\zeta=\zeta(\pi)=z_1z_2\dots z_{n-1}$ as defined in Section~\ref{sec:perm2word}. Let $c=\pi(n)$, $\ell=\pi^{-1}(n)$, and if $c\neq 1$, let $k=\pi^{-1}(c-1)$.
Define a polynomial $P_\pi(\beta)$ as follows.
 If $c=1$, let $$P_\pi(\beta)=p_{z_\ell z_{\ell+1}\dots z_{n-1}}(\beta);$$
if $c\neq 1$ and $\ell > k$, let $$P_\pi(\beta)=p_{z_\ell z_{\ell+1}\dots z_{n-1}z_k z_{k+1}\dots z_{\ell-1}}(\beta)-1;$$
if $c\neq 1$ and $\ell < k$, let $$P_\pi(\beta)=\begin{cases} p_{z_\ell z_{\ell+1}\dots z_{n-c}}(\beta) & \mbox{if $\pi$ ends in }12\dots c, \\
p_{z_\ell z_{\ell+1}\dots z_{n-1}}(\beta)-p_{z_\ell z_{\ell+1}\dots z_{k-1}}(\beta) & \mbox{otherwise}.\end{cases}$$
Then $\B(\pi)$ is the unique real root with $\beta\ge1$ of $P_\pi(\beta)$.
\end{theorem}

Note that $P_\pi(\beta)$ is always a monic polynomial with integer coefficients. For $\pi\in\S_n$, its degree is never greater than the maximum of $n-\ell$ and $n-k$, and in particular never greater than $n-1$.

\begin{proof}
In the case $c=1$, letting $w=\zeta0^\infty$, we know by Proposition~\ref{prop:c1} that $$\B(\pi)=\BW(w)=\BV(w_{\andon{\ell}})=\BV(z_\ell z_{\ell+1}\dots z_{n-1}0^\infty).$$
Thus, $\B(\pi)$ is the unique solution with $\beta\ge1$ of
$$\frac{z_\ell}{\beta}+\frac{z_{\ell+1}}{\beta^2}+\dots+\frac{z_{n-1}}{\beta^{n-\ell}}=1,$$
or equivalently, multiplying by $\beta^{n-\ell}$, of
$$\beta^{n-\ell}-z_\ell \beta^{n-\ell-1}-z_{\ell+1}\beta^{n-\ell-2}-\dots-z_{n-2}\beta-z_{n-1}=0,$$
that is, $p_{z_\ell z_{\ell+1}\dots z_{n-1}}(\beta)=0$.

In the case $c\neq 1$ and $\ell > k$, Theorem~\ref{th:findw1} states that if we now let
$$w = z_1z_2\cdots z_{n-1}\, z_{k} z_{k+1} \cdots z_{\ell-2}(z_{\ell-1}{+}1) 0^\infty,$$
then $\B(\pi)=\BW(w)=\BV(w_{\andon{\ell}})$.
Thus, $\B(\pi)$ is the unique solution with $\beta\ge1$ of
$$\frac{z_\ell}{\beta}+\frac{z_{\ell+1}}{\beta^2}+\dots+\frac{z_{n-1}}{\beta^{n-\ell}}+\frac{z_{k}}{\beta^{n-\ell+1}}+\dots+\frac{z_{\ell-2}}{\beta^{n-k-1}}+\frac{z_{\ell-1}+1}{\beta^{n-k}}=1,$$
or equivalently, multiplying by $\beta^{n-k}$, of
$$\beta^{n-k}-z_\ell \beta^{n-k-1}-z_{\ell+1}\beta^{n-k-2}-\dots-z_{n-1}\beta^{\ell-k}-z_{k}\beta^{\ell-k-1}-\dots-z_{\ell-1}-1=0,$$
that is, $p_{z_\ell z_{\ell+1}\dots z_{n-1}z_k z_{k+1}\dots z_{\ell-1}}(\beta)-1=0$.

Finally, if $c\neq 1$ and $\ell < k$, it follows from Theorem~\ref{th:findw1} that letting
$$w^{(m)} = z_1z_2\cdots z_{n-1}\, (z_{k} z_{k + 1} \cdots z_{n-1})^{m} z_k z_{k+1} \cdots z_{h-2}(z_{h-1}{+}1)0^\infty,$$
where $\pi(h)=\max\{\pi(k+1),\pi(k+2),\dots,\pi(n)\}$, we have
$$\B(\pi)=\lim_{m\rightarrow\infty}\BW(w^{(m)})$$
and $\BW(w^{(m)})=\BV(w^{(m)}_{\andon{\ell}})$.
Here $\BV(w^{(m)}_{\andon{\ell}})$ is the unique solution with $\beta\ge1$ of
\begin{multline}\label{eq:betam}\frac{z_\ell}{\beta}+\frac{z_{\ell+1}}{\beta^2}+\dots+\frac{z_{k-1}}{\beta^{k-\ell}}+\left(\frac{z_{k}}{\beta^{k-\ell+1}}+\dots+\frac{z_{n-1}}{\beta^{n-\ell}}\right)\left(1+\frac{1}{\beta^{n-k}}+\frac{1}{\beta^{2(n-k)}}+\dots+\frac{1}{\beta^{m(n-k)}}\right)
\\ +\frac{z_{k}}{\beta^{n-\ell+m(n-k)+1}}+\dots+\frac{z_{h-2}}{\beta^{n-\ell+m(n-k)+h-k-1}}+\frac{z_{h-1}+1}{\beta^{n-\ell+m(n-k)+h-k}}=1.\end{multline}
For fixed $m$, it is clear that $\BV(w^{(m)}_{\andon{\ell}})>1$, because $w^{(m)}_{\andon{\ell}}$ has at least two nonzero entries, since $z_\ell\ge1$.
Suppose first that not all of the entries $z_k,\dots,z_{n-1}$ are zero. In this case, making $m$ go to infinity in equation~(\ref{eq:betam})
and using that $\B(\pi)=\lim_{m\rightarrow\infty}\BV(w^{(m)}_{\andon{\ell}})$,
we see that $\B(\pi)$ is the solution with $\beta>1$ of
$$\frac{z_\ell}{\beta}+\frac{z_{\ell+1}}{\beta^2}+\dots+\frac{z_{k-1}}{\beta^{k-\ell}}+\left(\frac{z_{k}}{\beta^{k-\ell+1}}+\dots+\frac{z_{n-1}}{\beta^{n-\ell}}\right)\frac{1}{1-\frac{1}{\beta^{n-k}}}=1.$$
Multiplying by $\beta^{k-\ell}(\beta^{n-k}-1)$ we get
$$(\beta^{n-k}-1)(z_\ell\beta^{k-\ell-1}+z_{\ell+1}\beta^{k-\ell-2}+\dots+z_{k-1})+z_{k}\beta^{n-k-1}+\dots+z_{n-1}=\beta^{n-\ell}-\beta^{k-\ell},$$
which can be rearranged as
$$\beta^{n-\ell}-z_\ell\beta^{n-\ell-1}-z_{\ell+1}\beta^{n-\ell-2}-\dots-z_{n-1}=\beta^{k-\ell}-z_\ell\beta^{k-\ell-1}-z_{\ell+1}\beta^{k-\ell-2}-\dots-z_{k-1},$$
that is, $p_{z_\ell z_{\ell+1}\dots z_{n-1}}(\beta)=p_{z_\ell z_{\ell+1}\dots z_{k-1}}(\beta)$.

In the case where $z_k=\dots=z_{n-1}=0$, $\B(\pi)$ is the solution with $\beta\ge1$ of
$$z_\ell\beta^{k-\ell-1}+z_{\ell+1}\beta^{k-\ell-2}+\dots+z_{k-1}=\beta^{k-\ell},$$
or equivalently $p_{z_\ell z_{\ell+1}\dots z_{k-1}}(\beta)=0$.
This situation only happens when $\pi$ ends in $123\dots c$. Indeed,
one can use Lemma~\ref{lem:order} to show that the condition $z_k=\dots=z_{n-1}$ forces the sequence $\pi(k),\pi(k+1),\dots,\pi(n)$ to be monotonic,
which can only happen if $k=n-1$. Now, Lemma~\ref{lem:order} again and the fact that $z_k=0$ imply that if $d_i$ is the entry following $i$ in $\pi$, then
$1\neq d_1<d_2<\dots<d_{c-1}=c$, which forces the ending of $\pi$ to be $123\dots c$.
We remark that since $z_{n-c+1}=\dots=z_{n-1}=0$ in this case, we have that $p_{z_\ell z_{\ell+1}\dots z_{k-1}}(\beta)=\beta^{c-2}p_{z_\ell z_{\ell+1}\dots z_{n-c}}(\beta)$.
\end{proof}

\section{Examples}\label{sec:examples}

In this section we give examples where Proposition~\ref{prop:c1} and Theorems~\ref{th:findw1}, \ref{th:findw2}, and~\ref{th:Bpi} are used to construct words inducing a given permutation and to determine its shift-complexity.

\ben \renewcommand{\labelenumi}{(\arabic{enumi})}
\item Let $\pi=3421$. Using the construction from~\cite{Elishifts}, described also right after Proposition~\ref{prop:determined} above,
we get $\zeta(\pi)=121$. Proposition~\ref{prop:c1} states that $w=1210^\infty$ induces $\pi$ and
$\B(\pi)=\BW(w)=\BV(210^\infty)$. By Theorem~\ref{th:Bpi}, $\B(\pi)$
is the root with $\beta\ge1$ of
$$P_\pi(\beta)=p_{21}(\beta)=\beta^2-2\beta-1,$$
so $\B(3421)=1+\sqrt{2}$.

\item Let $\pi=735491826$. Using the construction from~\cite{Elishifts},
$\zeta(\pi)=42326051$. Applying Theorem~\ref{th:findw1} with $k=3$ and $\ell=5$, we get that $w=42326051330^\infty$ induces $\pi$ and
$$\B(\pi)=\BW(w)=\BV(6051330^\infty).$$
By Theorem~\ref{th:Bpi}, $\B(\pi)$
is the real root with $\beta\ge1$ of
$$P_\pi(\beta)=p_{605133}(\beta)-1=\beta^6-6\beta^5-5\beta^3-\beta^2-3\beta-3,$$
so $\B(735491826)\approx6.139428921$.

\item For $\pi=2516437$, we get $\zeta(\pi)=1303213$. By Theorem~\ref{th:findw1} with $k=7$ and $\ell=8$, the word $w=130321340^\infty$ induces $\pi$ and
$\B(\pi)=\BV(40^\infty)=4$.
In this simple case, $P_\pi(\beta)=p_{3}(\beta)-1=\beta-4$.

\item For $\pi=892364157$, we have seen earlier that $\zeta(\pi)=34113202$. Applying Theorem~\ref{th:findw2} with $k=5$, $\ell=2$, and $h=9$,
we have that $$w^{(m)}=34113202(3202)^m32030^\infty$$ induces $\pi$ for $m\ge2$, and
$$\B(\pi)=\lim_{m\rightarrow\infty} \BW(w^{(m)})=\lim_{m\rightarrow\infty} \BV(4113202(3202)^m32030^\infty).$$
By Theorem~\ref{th:Bpi}, $\B(\pi)$
is the real root with $\beta\ge1$ of
\begin{multline*}P_\pi(\beta)=p_{4113202}(\beta)-p_{411}(\beta)=(\beta^7-4\beta^6-\beta^5-\beta^4-3\beta^3-2\beta^2-2)-(\beta^3-4\beta^2-\beta-1)\\
=\beta^7-4\beta^6-\beta^5-\beta^4-4\beta^3+2\beta^2+\beta-1,\end{multline*}
so $\B(892364157)\approx4.327613926$.

\item For $\pi=85132674$, we get $\zeta(\pi)=3201023$. By Theorem~\ref{th:findw2} with $k=4$, $\ell=1$, and $h=7$, we have that
$$w^{(m)}=3201023(1023)^m1030^\infty$$ induces $\pi$ for $m\ge2$, and $$\B(\pi)=\lim_{m\rightarrow\infty} \BV(3201023(1023)^m1030^\infty).$$
By Theorem~\ref{th:Bpi}, $\B(\pi)$
is the real root with $\beta\ge1$ of
\begin{multline*}P_\pi(\beta)=p_{3201023}(\beta)-p_{320}(\beta)=(\beta^7-3\beta^6-2\beta^5-\beta^3-2\beta-3)-(\beta^3-3\beta^2-2\beta)\\
=\beta^7-3\beta^6-2\beta^5-2\beta^3+3\beta^2-3,\end{multline*}
so $\B(892364157)\approx3.584606864$.

\item Let $\pi=(c+1)(c+2)\dots n 1 2 \dots c$ for any fixed $1\le c\le n$. Here we get $\zeta(\pi)=0^{n-c-1}10^{c-1}$.
If $1<c<n$, then $k=n-1$, $\ell=n-c$ and $h=n$,
so by Theorem~\ref{th:findw2}, $$w^{(m)}=0^{n-c-1}10^{c-1}0^m10^\infty$$ induces $\pi$ for $m\ge n-2$, and $$\B(\pi)=\lim_{m\rightarrow\infty} \BV(10^{c-1}0^m10^\infty).$$
By Theorem~\ref{th:Bpi}, $\B(\pi)=1$ is the root of $P_\pi(\beta)=p_{1}(\beta)=\beta-1$.

If $c=n$, Theorem~\ref{th:findw1} gives $w=0^{n-1}10^\infty$, and
if $c=1$, Proposition~\ref{prop:c1} yields $w=0^{n-2}10^\infty$.
In both cases, $w$ induces $\pi$ and $\B(\pi)=\BV(10^\infty)=1$ as well.

It is not hard to see that these are the only permutations with $\B(\pi)=1$.
\een

The values of $\B(\pi)$ for all permutations of length 2, 3, and 4 are given in Table~\ref{tab:Bpi}. For permutations of length 5, these values appear in Table~\ref{tab:Bpi5}. They have been computed using
the implementation in {\it Maple} of the algorithm described in Sections~\ref{sec:perm2word} and~\ref{sec:Bpi}.

\begin{table}[hbt]
$$\begin{array}{|c|c|c|c|c|}
\hline
  \pi\in\S_2 &\pi\in\S_3 & \pi\in\S_4 & \B(\pi) & \B(\pi)\mbox{ is a root of}\\ \hline
  12, 21 & 123,231,312 & 1234,2341,3412,4123 &  1 & \beta-1\\ \hline
  & & 1342,2413,3124,4231 &  1.465571232 & \beta^3-\beta^2-1 \\ \hline
  & 132,213,321 & 1243,1324,2431,3142,4312 &  \frac{1+\sqrt{5}}{2}\approx1.618033989 & \beta^2-\beta-1  \\ \hline
  & & 4213 &  1.801937736 & \beta^3-\beta^2-2\beta+1 \\ \hline
  & & 1432,2143,3214,4321 &  1.839286755 & \beta^3-\beta^2-\beta-1 \\ \hline
  & & 2134,3241 &  2 & \beta-2 \\ \hline
  & & 4132 &  2.246979604 & \beta^3-2\beta^2-\beta+1 \\ \hline
  & & 2314,3421 &  1+\sqrt{2}\approx2.414213562 & \beta^2-2\beta-1 \\ \hline
  & & 1423 &  \frac{3+\sqrt{5}}{2}\approx2.618033989 & \beta^2-3\beta+1 \\ \hline
\end{array}$$
\caption{The shift-complexity of all permutations of length up to 4.}\label{tab:Bpi}
\end{table}

\begin{table}[hbt] 
{\small $$\begin{array}{|c|c|c|}
\hline
\pi\in\S_5 & \B(\pi) & \B(\pi)\mbox{ is a root of}\\ \hline
 12345,23451,34512,45123,51234 & 1 & \beta-1\\ \hline
 13452,24513,35124, 41235,52341 & 1.380277569 & \beta^4-\beta^3-1 \\ \hline
 12453,13524,24135, 35241,41352,53412& 1.465571232 & \beta^3-\beta^2-1 \\ \hline
 52413 & 1.558979878 & \beta^4-\beta^3-2\beta+1 \\ \hline
 \ba{c} 12354,12435,14253,23541,31425,\\ 35412,41253,42531,54123\ea & \frac{1+\sqrt{5}}{2}\approx1.6180 & \beta^2-\beta-1  \\ \hline
 53124 & 1.722083806 & \beta^4-\beta^3-\beta^2-\beta+1 \\ \hline
 13542,25413,31254, 43125,54231 & 1.754877666 &
 \beta^3-2\beta^2+\beta-1 \\ \hline
25314,53142 & 1.801937736 & \beta^3-\beta^2-2\beta+1 \\ \hline
 12543,13254,14325, 25431, 31542,42153,54312 & 1.839286755 & \beta^3-\beta^2-\beta-1 \\ \hline
54213 & 1.905166168 & \beta^4-\beta^3-2\beta^2+1 \\ \hline
 53214 & 1.921289610 & \beta^4-\beta^3-\beta^2-2\beta+1 \\ \hline
 15432,21543,32154,43215,54321 & 1.927561975 & \beta^4-\beta^3-\beta^2-\beta-1 \\ \hline
 \ba{c} 13245,21345,24351,31245,32145,\\ 32451,42351,43251,43512\ea & 2 & \beta-2 \\ \hline
 51342 & 2.117688633 & \beta^4-2\beta^3-\beta+1 \\ \hline
 51243 & \frac{1+\sqrt{5+4\sqrt{2}}}{2}\approx2.1322 & \beta^4-2\beta^3-\beta^2+2\beta-1 \\ \hline
 34125,42513,45231 & 2.205569430 & \beta^3-2\beta^2-1 \\ \hline
35142,45132,51324 & 2.246979604 & \beta^3-2\beta^2-\beta+1 \\ \hline
 14352,25143,32514,41325,52431  & 2.277452390 & \beta^4-2\beta^3-\beta-1 \\ \hline
 51432 & 2.296630263 & \beta^4-2\beta^3-2\beta+1 \\ \hline
 25134 & 2.324717957 & \beta^3-3\beta^2+2\beta-1 \\ \hline
 23514,31452 & 2.359304086 & \beta^3-2\beta^2-2 \\ \hline
 13425,23415,24531,34152,34521,43152,45312 & 1+\sqrt{2}\approx2.4142 & \beta^2-2\beta-1 \\ \hline
 45213 & 2.481194304 & \beta^3-2\beta^2-2\beta+2 \\ \hline
52143 & 2.496698205 & \beta^4-2\beta^3-\beta^2-\beta+1 \\ \hline
52134 & 2.505068414 & \beta^4-3\beta^3+\beta^2+\beta-1 \\ \hline
 14532,21453,35214,42135,53241  & 2.521379707 & 
 \beta^3-3\beta^2+2\beta-2 \\ \hline
 34215,41532,45321 & 2.546818277 & \beta^3-2\beta^2-\beta-1 \\ \hline
12534,14523,15234,21534,41523 & \frac{3+\sqrt{5}}{2}\approx2.6180 & \beta^2-3\beta+1 \\ \hline
 14235,25341 & 2.658967082 & \beta^3-2\beta^2-\beta-2 \\ \hline
 52314 & 2.691739510 & \beta^4-2\beta^3-2\beta^2+1 \\ \hline
15342,24153,31524,42315,53421 & 2.696797189 & \beta^4-2\beta^3-\beta^2-2\beta-1 \\ \hline
 21354,21435,32541 & 1+\sqrt{3}\approx2.7320 & \beta^2-2\beta-2 \\ \hline
 54132 & 2.774622899 & \beta^4-2\beta^3-3\beta^2+2\beta+1 \\ \hline
23154,24315,35421 & 2.831177207 & \beta^3-2\beta^2-2\beta-1 \\ \hline
 15423 & 2.879385242 & \beta^3-3\beta^2+1 \\ \hline
15324 & 2.912229178 & \beta^3-2\beta^2-3\beta+1 \\ \hline
23145,34251 & 3 & \beta-3 \\ \hline
51423 & 3.234022893 & \beta^4-4\beta^3+3\beta^2-2\beta+1 \\ \hline
32415,43521 & \frac{3+\sqrt{13}}{2}\approx3.3028 & \beta^2-3\beta-1 \\ \hline
15243 & 3.490863615 & \beta^3-3\beta^2-2\beta+1 \\ \hline
\end{array}$$}
\caption{The shift-complexity of all permutations of length 5. }\label{tab:Bpi5}
\end{table}

\section{The shortest forbidden pattern of $\Sigma_\beta$}\label{sec:shortest}

In the previous two sections, our goal was to compute the smallest $\beta$ needed for a given permutation to be realized by the $\beta$-shift.
In this section we consider the reverse problem: given a real number $\beta>1$, we want to determine the length of the shortest forbidden pattern of $\Sigma_\beta$.
This is useful in practice when discriminating between sequences generated by $\beta$-shifts from random sequences by looking for missing patterns.

When $\beta=N\ge2$ is an integer, Theorem~\ref{th:aek_minshift} implies that the length of the shortest forbidden pattern of $\Sigma_\beta$ is $n=N+2$, and
Proposition~\ref{prop:6special} states that there are exactly six forbidden patterns of shortest length $n$.
Let us denote this set by $$\G_n=\{\rho,\rho^R,\rho^C,\rho^{RC},\tau,\tau^C\}.$$ 
Recall that if $n$ is even and we let $s=n/2$, then
\begin{eqnarray*}
\rho&=& 1\, n\, 2\, (n{-}1)\, 3\, (n{-}2)\, \dots\, (s{-}1)\,(s{+}2)\,s\,(s{+}1),\\
\rho^R&=& (s{+}1)\,s \, (s{+}2)\, (s{-}1)\,  \dots\,(n{-}2)\,3\, (n{-}1) \,2\,n\,1,\\
\rho^C&=& n\, 1\, (n{-}1)\, 2\, (n{-}2)\, 3\, \dots\, (s{+}2)\,(s{-}1)\,(s{+}1)\,s,\\
\rho^{RC}&=& s\, (s{+}1)\,(s{-}1) \, (s{+}2)\,   \dots\,3\,(n{-}2)\,2 \,(n{-}1)\,1\,n,\\
\tau&=&(s{+}1)\,(s{+}2)\,s\,(s{+}3)\,\dots\, 4\, (n{-}1)\, 3\, n\, 2\, 1,\\
\tau^C&=&s\,(s{-}1)\,(s{+}1)\,(s{-}2)\,\dots\, (n{-}3)\, 2\,(n{-}2)\, 1\, (n{-}1)\, n,
\end{eqnarray*}
and if $n$ is odd and we let $s=(n+1)/2$, then
\begin{eqnarray*}
\rho&=& 1\, n\, 2\, (n{-}1)\, 3\, (n{-}2)\, \dots\, (s{+}2)\,(s{-}1)\,(s{+}1)\,s,\\
\rho^R&=& s \,(s{+}1)\, (s{-}1)\,(s{+}2)\, \dots\,(n{-}2)\,3\, (n{-}1) \,2\,n\,1,\\
\rho^C&=& n\, 1\, (n{-}1)\, 2\, (n{-}2)\, 3\, \dots\, (s{-}2)\,(s{+}1)\,(s{-}1)\,s,\\
\rho^{RC}&=& s\, (s{-}1)\,(s{+}1) \, (s{-}2)\,   \dots\,3\,(n{-}2)\,2 \,(n{-}1)\,1\,n,\\
\tau&=&(s{+}1)\,s\,(s{+}2)\,(s{-}1)\,\dots\, 4\, (n{-}1)\, 3\, n\, 2\, 1,\\
\tau^C&=&(s{-}1)\,s\,(s{-}2)\,(s{+}1)\,\dots\, (n{-}3)\, 2\,(n{-}2)\, 1\, (n{-}1)\, n.
\end{eqnarray*}

It will be convenient to extend the definition to $n=3$, which gives $\Gamma_3=\S_3$, and to define $\rho=12$ when $n=2$.
We can rephrase Proposition~\ref{prop:6special} in terms of the statistic $N(\pi)$, defined in equation~(\ref{def:N}), as follows.
\begin{proposition}\label{prop:6special2}
Let $n\ge3$, and let $\pi\in\S_n$. We have
$N(\pi)=n-1$ if $\pi\in\G_n$, and $N(\pi)\le n-2$ otherwise.
\end{proposition}

We now use the techniques developed in Sections~\ref{sec:perm2word} and~\ref{sec:Bpi} to compute the shift-complexity of the six permutations in $\Gamma_n$.

\begin{proposition}\label{prop:Brho}
Let $n\ge4$, and let $\pi\in\Gamma_n$. Then $\B(\pi)$ is the unique real solution with $\beta>1$ of the equation $\beta=F_\pi(\beta)$, where
\bea\nn F_\rho(\beta)&=&\begin{cases} n-2+\dfrac{1}{\beta}+\dfrac{1}{\beta+1}-\dfrac{1}{\beta^{n-2}(\beta+1)} & \mbox{if $n$ is even}, \\
n-2+\dfrac{1}{\beta}+\dfrac{1}{\beta+1}-\dfrac{1}{\beta^{n-3}(\beta+1)} & \mbox{if $n$ is odd},\end{cases} \\
\nn F_{\rho^{RC}}(\beta)=F_{\tau}(\beta)&=& n-2+\frac{1}{\beta},\\
\nn F_{\rho^C}(\beta)&=&\begin{cases} n-2+\dfrac{1}{\beta+1}-\dfrac{1}{\beta^{n-2}(\beta+1)} & \mbox{if $n$ is even}, \\
n-2+\dfrac{1}{\beta+1}-\dfrac{1}{\beta^{n-1}(\beta+1)} & \mbox{if $n$ is odd},\end{cases} \\
\nn F_{\rho^{R}}(\beta)=F_{\tau^C}(\beta)&=& n-2.
\eea
\end{proposition}

For $n=3$, it is easy to check using Theorem~\ref{th:Bpi} that
$\B(132)=\B(213)=\B(321)=\frac{1+\sqrt{5}}{2}$ and $\B(123)=\B(231)=\B(312)=1$, which in fact coincide with the solutions with $\beta\ge1$ of the equations in Proposition~\ref{prop:Brho} for $n=3$.
For $n=2$, clearly $\B(12)=\B(21)=1$.

\begin{proof}
Let $s=\lceil n/2 \rceil$. We start with $\pi=\rho$. If $n$ is even, the word given by Theorem~\ref{th:findw2} is
$$w^{(m)}=0(n{-}2)1(n{-}3)2(n{-}4)\dots(s{+}1)(s{-}2)s(s{-}1)(s{-}1)^m s0^\infty,$$
and Theorem~\ref{th:Bpi} states that $\B(\rho)$ is the unique root with $\beta>1$ of
\begin{eqnarray*}
P_\rho(\beta)&=&p_{(n{-}2)1(n{-}3)2(n{-}4)\dots(s{+}1)(s{-}2)s(s{-}1)}(\beta)-p_{(n{-}2)1(n{-}3)2(n{-}4)\dots(s{+}1)(s{-}2)s}(\beta)\\
&=&\beta^{n-2}-(n-1)\beta^{n-3}+(n-3)\beta^{n-4}-(n-4)\beta^{n-5}+(n-5)\beta^{n-6}-\dots-2\beta+1\\
&=&\beta^{n-2}-\beta^{n-3}-\dfrac{(n-2)\beta^{n-1}+(n-1)\beta^{n-2}-1}{(\beta+1)^2}.\end{eqnarray*} After some algebraic manipulations,
the equation $P_\rho(\beta)=0$ becomes
$$\beta=n-2+\frac{1}{\beta}+\frac{1}{\beta+1}-\frac{1}{\beta^{n-2}(\beta+1)}.$$
If $n$ is odd, using that $n\ne3$, Theorem~\ref{th:findw2} gives
$$w^{(m)}=0(n{-}2)1(n{-}3)2(n{-}4)\dots(s{-}3)s(s{-}2)(s{-}1)((s{-}2)(s{-}1))^m(s{-}1)0^\infty,$$
so by Theorem~\ref{th:Bpi}, \begin{eqnarray*}P_\rho(\beta)&=&p_{(n{-}2)1(n{-}3)2(n{-}4)\dots(s{-}3)s(s{-}2)(s{-}1)}(\beta)-p_{(n{-}2)1(n{-}3)2(n{-}4)\dots(s{-}3)s}(\beta)\\
&=&\beta^{n-2}-(n-2)\beta^{n-3}-2\beta^{n-4}+\beta^{n-5}-\beta^{n-6}+\beta^{n-7}-\dots-\beta+1 \\
&=&\beta^{n-2}-(n-2)\beta^{n-3}-\beta^{n-4}-\dfrac{\beta^{n-3}-1}{\beta+1}.\end{eqnarray*}
The equation $P_\rho(\beta)=0$ can be written as
\beq\label{eq:Prho}\beta=n-2+\frac{1}{\beta}+\frac{1}{\beta+1}-\frac{1}{\beta^{n-3}(\beta+1)}.\eeq

For $\pi\in\{\rho^{RC},\tau\}$, after some computations, Theorem~\ref{th:Bpi} gives the polynomial
$$P_{\rho^{RC}}(\beta)=P_{\tau}(\beta)=\beta^2-(n-2)\beta-1,$$
from where the equation follows. In fact, $$\B(\rho^{RC})=\B(\tau)=\frac{n-2+\sqrt{(n-2)^2+4}}{2}.$$

For $\pi=\rho^C$, a similar argument shows that if $n$ is even,
\begin{eqnarray*}
P_{\rho^C}(\beta)&=&\beta^{n-1}-(n-2)\beta^{n-2}-\beta^{n-3}+\beta^{n-4}-\beta^{n-5}-\dots-\beta+1\\
&=&\beta^{n-1}-(n-2)\beta^{n-2}-\dfrac{\beta^{n-2}-1}{\beta+1},\end{eqnarray*}
and if $n$ is odd,
\begin{eqnarray*}
P_{\rho^C}(\beta)&=&\beta^{n-1}-(n-1)\beta^{n-2}+(n-2)\beta^{n-3}-(n-3)\beta^{n-4}-\dots-2\beta+1\\
&=&\beta^{n-1}-\dfrac{(n-1)\beta^{n}+n\beta^{n-1}-1}{(\beta+1)^2}.\end{eqnarray*}
The equations for $\B(\rho^C)$ are obtained by setting $P_{\rho^C}(\beta)=0$.

Finally, for $\pi\in\{\rho^R,\tau^C\}$, we get $P_{\rho^{R}}(\beta)=P_{\tau^C}(\beta)=\beta-(n-2)$, so $$\B(\rho^{R})=\B(\tau^C)=n-2.$$
\end{proof}

It is a consequence of the above result that the maximum shift-complexity for permutations in~$\G_n$ is achieved at~$\rho$.
\begin{corollary}\label{cor:rhoG}
Let $n\ge4$, and let $\pi\in\G_n\setminus\{\rho\}$. Then $\B(\pi)<\B(\rho)$.
\end{corollary}

\begin{proof}
It follows from Proposition~\ref{prop:Brho} that for all $\beta>1$,
$$F_{\rho^{R}}(\beta)=F_{\tau^C}(\beta)<F_{\rho^C}(\beta)<F_{\rho^{RC}}(\beta)=F_{\tau}(\beta)<F_{\rho}(\beta).$$
For each $\pi\in\Gamma_n$, $\B(\pi)$ is the unique intersection with $\beta>1$ of the graph of $F_\pi(\beta)$ with the line $y=\beta$.
Since $\lim_{\beta\rightarrow\infty}F_\pi(\beta)=0$, we have
$$\B(\rho^{R})=\B(\tau^C)<\B(\rho^C)<\B(\rho^{RC})=\B(\tau)<\B(\rho).$$
\end{proof}

Now we come to the main result of this section, namely that among all permutations of length~$n$, $\rho$ is the one with the highest shift-complexity.

\begin{theorem}\label{thm:rhoS}
Let $n\ge4$, and let $\pi\in\S_n\setminus\{\rho\}$. Then $$\B(\pi)<\B(\rho).$$
\end{theorem}

Note that for $n\in\{2,3\}$ and $\pi\in\S_n\setminus\{\rho\}$, we have $\B(\pi)\le\B(\rho)$.

\begin{proof}
We know by Proposition~\ref{prop:6special2} that if $\pi\in\S_n\setminus\G_n$, then $N(\pi)\le n-2$. Thus, by equation~(\ref{eq:BN}), $\fl{\B(\pi)}\le n-3$, so $$\B(\pi)<n-2$$ in this case.
On the other hand, if $\pi\in\G_n$, then $\fl{\B(\pi)}=N(\pi)-1=n-2$, so $$n-2\le \B(\pi)<n-1.$$
Thus, the six permutations in $\G_n$ have a higher value of $\B(\pi)$ than permutations in $\S_n\setminus\G_n$,
and among these six, $\pi=\rho$ gives the highest value of $\B(\pi)$ by Corollary~\ref{cor:rhoG}.
\end{proof}

For each $n\ge2$, let $\beta_n=\B(\rho)$, where $\rho\in\S_n$. We note that $n-2<\beta_n<n-1$ for $n\ge3$. The first terms of the sequence $\{\beta_n\}_{n\ge2}$ are (up to truncation)
$1$, $1.618033989$, $2.618033989$, $3.490863615$, $4.411024434$, $5.344530094$, $6.295894835$, $7.258844460$, $8.229852937$\dots.

It follows from Theorem~\ref{thm:rhoS} and Proposition~\ref{prop:open} that $\S_n\subseteq\Al(\Sigma_\beta)$ if and only if $\beta>\beta_n$. In other words,
$\beta_n$ is the threshold after which all permutations in $\S_n$ are realized by the $\beta$-shift.
It is now straightforward to determine the length of the shortest forbidden pattern of $\Sigma_\beta$.
The following is a generalization of Theorem~\ref{th:aek_minshift}.

\begin{theorem}
Let $\beta>1$. The length of the shortest forbidden pattern of $\Sigma_\beta$ is the value of $n$ such that
$\beta_{n-1}<\beta\le \beta_n$.
\end{theorem}

\begin{proof}
We have seen that $\S_m\subseteq\Al(\Sigma_\beta)$ if and only if $\beta>\beta_m$. Thus, if $\beta_{n-1}<\beta\le \beta_n$, then $\S_{n-1}\subseteq\Al(\Sigma_\beta)$ but $\S_n\nsubseteq\Al(\Sigma_\beta)$,
so the shortest forbidden pattern of $\Sigma_\beta$ has length $n$.
\end{proof}

From the equation in Proposition~\ref{prop:Brho} satisfied by $\beta_n$, namely
$$\beta_n=n-2+\frac{1}{\beta_n}+\frac{1}{\beta_n+1}-\frac{1}{\beta_n^{n-2-\delta}(\beta_n+1)},$$
where $\delta=1$ ($\delta=0$) if $n$ is odd (even), and the fact that $n-2<\beta_n<n-1$ for $n\ge 3$, we obtain the asymptotic growth of $\beta_n$ as $n$ goes to infinity:
\beq\label{eq:betan} \beta_n=n-2+\frac{2}{n}+O(\frac{1}{n^2}).\eeq
In particular, $\beta_n$ is close to $n-2$ for large $n$. Similarly, one can show that for $\pi\in\{\rho^C,\rho^{RC},\tau\}$,
$$\B(\pi)=n-2+\frac{1}{n}+O(\frac{1}{n^2}).$$

Several questions arise when looking at the values of $\B(\pi)$ where $\pi$ ranges over all permutations. One of them
is to describe which algebraic numbers are obtained in this way, and what are the accumulation points. It is not hard to see, for example,
that positive integers are accumulation points. Another interesting question is how many permutations of length $n$ have $\B(\pi)<\beta$ for a fixed $\beta$. This is equivalent to counting the allowed patterns of the $\beta$-shift.

\subsection*{Acknowledgement} The author thanks Alex Borland for significant contributions, including several results in Section~\ref{sec:perm2word}, as part of his undergraduate research experience
funded by a Presidential Scholarship
from Dartmouth College.

\end{document}